%% file: main.tex
\documentclass[onecolumn, 11pt, letterpaper]{article}

\usepackage[margin = 1in]{geometry}
\usepackage{amssymb}
\usepackage{algorithm}
\usepackage{algorithmic}
\usepackage{bm}
\usepackage{dsfont}
\usepackage{bbm}
\usepackage{xcolor}
\usepackage{amsthm}
\usepackage{amsmath}

\usepackage{cite}

\usepackage{graphicx}

\newtheorem{corollary}{Corollary}
\usepackage[utf8]{inputenc}
\usepackage[english]{babel}
 
\newtheorem{theorem}{Theorem}
\newtheorem{remark}{\textbf{Remark}}

\usepackage{authblk}

\usepackage{algorithm,algorithmic}
\usepackage{caption}

\usepackage{subcaption}

\newlength\myindent
\usepackage{enumitem,kantlipsum}

\usepackage{subfloat}

\usepackage{cite}
\usepackage{graphicx}
\usepackage{mathtools}
\usepackage{amsmath,amssymb,amsfonts}
\usepackage{mathrsfs}
\usepackage[hidelinks]{hyperref}
\usepackage{graphics}
\usepackage{lipsum, color}
\usepackage{epstopdf}

\usepackage{hhline}
\usepackage{booktabs}
\usepackage{makecell}
\usepackage{multirow}
\usepackage{dcolumn}
\usepackage{blindtext}  \usepackage{verbatim} \usepackage{listings} \usepackage{color}    \usepackage{tikz}

\usetikzlibrary{shapes,arrows}
\usetikzlibrary{positioning}
\usetikzlibrary{shapes,arrows}
\usetikzlibrary{arrows,calc, matrix}
\tikzset{
    block/.style = {draw, rectangle,
        minimum height=1cm,
        minimum width=1cm},
    input/.style = {coordinate,node distance=1.3cm},
    output/.style = {coordinate,node distance=2.3cm},
    arrow/.style={draw, -latex,node distance=2cm},
    pinstyle/.style = {pin edge={latex-, black,node distance=2cm}},
    sum/.style = {draw, circle, node distance=1cm}
}

\newtheorem{assumption}{Assumption}
\usepackage{bbm}

\def\Var{{\rm Var}\,}

\usepackage{mathtools}
\newcommand{\bR}{\mathbb{R}}
\newcommand{\bE}{\mathbb{E}}

\newcommand{\bP}{\mathbb{P}}

\DeclarePairedDelimiter\abs{\lvert}{\rvert}\DeclarePairedDelimiter\norm{\lVert}{\rVert}\makeatletter
\let\oldabs\abs
\def\abs{\@ifstar{\oldabs}{\oldabs*}}
\let\oldnorm\norm
\def\norm{\@ifstar{\oldnorm}{\oldnorm*}}
\usepackage{amsmath}
\DeclareMathOperator*{\argmax}{argmax} % thin space, limits underneath in displays
\DeclareMathOperator*{\argmin}{argmin} % thin space, limits underneath in displays

\newtheorem{definition}{Definition}

\usepackage[switch, pagewise, displaymath, mathlines]{lineno}

\title{\LARGE \bf
A Hitting Time Analysis for Stochastic Time-Varying Functions with Applications to Adversarial Attacks on Computation of Markov Decision Processes
}

\author{Ali Yekkehkhany$^1$, Han Feng$^1$, Donghao Ying$^1$, Javad Lavaei}
\affil{
Department of Industrial Engineering and Operations Research -
University of California, Berkeley
}
\affil[$\ $]{\{aliyek, han\_feng, donghaoy, lavaei\}@berkeley.edu}

\date{}                     %% if you don't need date to appear
\providecommand{\keywords}[1]
{
  \small	
  \textbf{\textit{Keywords---}} #1
}

\begin{document}
\maketitle
% \thispagestyle{empty}
% \pagestyle{plain}
%%%COMMENT%%%\linenumbers
%%%%%%%%%%%%%%%%%%%%%%%%%%%%%%%%%%%%%%%%%%%%%%%%%%%%%%%%%%%%%%%%%%%%%%%%%%%%%%%%
\begin{abstract}
{Stochastic time-varying optimization is an integral part of learning in which the shape of the function changes over time in a non-deterministic manner. This paper considers multiple models of stochastic time variation and analyzes the corresponding notion of hitting time for each model, i.e.,  the period after which optimizing the stochastic time-varying function reveals informative statistics on {the} optimization of the target function. 
The studied models of time variation are motivated by adversarial attacks on the computation of value iteration in Markov decision processes. In this application, the hitting time quantifies the extent that the computation is robust to adversarial disturbance. We {develop} upper bounds {on the} hitting time by analyzing the contraction-expansion transformation appeared in the time-variation models. We prove that the hitting time of the value function in the value iteration with a probabilistic contraction-expansion transformation is logarithmic in terms of the inverse of a desired precision.} {In addition,} the hitting time is analyzed for optimization of unknown {continuous or discrete} time-varying functions whose noisy evaluations are revealed over time.
{The upper bound for a continuous function is super-quadratic (but sub-cubic) in terms of the inverse of a desired precision and the upper bound for a discrete function is logarithmic in terms of the cardinality of the function domain.} Improved bounds for convex functions are obtained and we show that such functions are learned faster than non-convex functions. Finally, we study a time-varying linear model with additive noise, where hitting time is bounded with the notion of shape dominance. 

%{In this application,} the contraction of value iteration is reversed to an expansion up to a constant by an adversary in a probabilistic manner. }
% In particular, are studied with  using a probabilistic Banach fixed-point theorem  
% , violating the contraction {criterion} of the Banach fixed-point theorem.
% To {address the problem}, 
% in which {the} convergence of {the} value iteration with a probabilistic contraction-expansion transformation is proved with an associated confidence level.

\end{abstract}
\keywords{Stochastic time-varying functions, stochastic operators, hitting time, probabilistic contraction-expansion mapping, probabilistic Banach fixed-point theorem, adversarial Markov decision process}
% \textit{Keywords}: Stochastic time-varying functions, stochastic operators, hitting time, probabilistic contraction-expansion mapping, probabilistic Banach fixed-point theorem, adversarial Markov decision process.

\input{sec/intro}
\input{sec/continuous}
\input{sec/discrete}

\input{sec/simulation}

\input{sec/conclusion}

\bibliographystyle{unsrt}
\bibliography{sigproc}

\end{document}

%% file: sec/intro.tex
% Text of your paper here
\section{Introduction and Related Work}
\label{introduction}
In many practical applications of optimization, such as those in the training of neural networks~\cite{sun2019optimization,gu2020implicit}, online advertising~\cite{bottou2013counterfactual}, decision-making process of power systems~\cite{mulvaney2020load,park2020homotopy}, and the real-time state estimation of nonlinear systems~\cite{rao2003constrained}, the parameters of the problem are often uncertain and change over time \cite{ajalloeian2020inexact}.
%%%COMMENT%%%\blfootnote{A preliminary treatment of the problem discussed in this {paper} has been introduced in our earlier work that is submitted to the American Control Conference 2021 \cite{fenghitting}.
%%%COMMENT%%%Compared with the conference version whose focus is on discrete time-varying functions, this {paper includes} multiple new models of time variation in continuous domains.
%%%COMMENT%%%The new studies in this {paper include} the analysis of adversarial attacks on {the} computation of Markov decision processes, which {also} distinguishes this {work} from our conference paper.}
To put the time-varying and uncertainty of the systems into perspective in optimization problems, time-varying {or online optimization aims} to find the solution trajectories determined by
\begin{align}
x^*_t = \argmin_{x \in \mathcal{X}} \left\{  f_t(x) = \bE F_t(x, \xi) \right\}, \quad t \in \{1, 2, \dots \} \label{eq:varyingform},
\end{align}
where the random variable $\xi$ models the uncertainty in the objective that comes from disturbance, inexactness of model, use of small batches, or injected noise, {and where $\argmin$ denotes any global minimizer of the input function}.
Note that the expectation $\bE$ over $\xi$ can only be evaluated approximately since the probability distribution is unknown, {and therefore} the target function $f_t$ should be approximated by observed samples.
The estimate of the target function may not capture the shape of the target function given {a} limited number of observed samples. {However,} there is a point of time, {named} \emph{hitting time}, after which optimizing the estimated target function results in optimizing the target function up to some precision and confidence level.
The hitting time captures the stochastic complexity of the time-varying problem in \eqref{eq:varyingform}. %, which is studied for multiple models in this {work}.
%Interestingly, we find out that convex functions are generally learned faster, which is analogous to the well-known fact that convex functions have a lower optimization complexity.
%{The focus of this paper is} on the stochastic complexity of stochastic time-varying functions, captured by the hitting time notion. 
%{For} an optimization-based study of time-varying problems, the {reader is} referred to the works by \cite{simonetto2016class, tang2018running, fattahi2020absence, ding2021escaping} and the references therein.

        \begin{table*}[htb]
        \centering
        \caption{\label{tab:table1}Comparison of Selected Theorems in Sections II-III}
         {\begin{tabular}{ccc}
        \toprule
        	Theorem	&  Assumptions & Hitting Time Definition \\
        \midrule
        	  \ref{theorem_hitting_time}  & Assumptions \ref{assump:noise_N}-\ref{assump:granularity}, bounded difference functions & \eqref{hitting_time} \\ 
        	 \ref{theorem_hitting_time_convex}  & Assumptions \ref{assump:noise_N}-\ref{assump:gradient_lower_bound}, convex bounded difference functions  & \eqref{hitting_time} \\
        	 \ref{theorem_discrete_upper_bound}  & Assumptions \ref{assump:noise_N} and \ref{assump:minimum_distance}  & \eqref{hitting_time2} \\
        	 \ref{theorem_unimodal_discrete}  & Assumptions \ref{assump:noise_N} and \ref{assump:minimum_distance}, unimodal functions & \eqref{hitting_time2} \\
         	 \ref{thm:hitting}	& linear dynamics and shape dominance & \eqref{eq:hitting-time-def} \\
        \bottomrule
        \end{tabular}}
        \end{table*}
\subsection{Motivating Applications}
In order to motivate the analysis of hitting time for time-varying probabilistic transformations, {we first explain its} applications in Markov Decision Process (MDP) and reinforcement learning (RL).
Consider an MDP with the set of states (state space) $\mathcal{S}$, the set of actions (action space) $\mathcal{A}$, the time-invariant state transition $h$ such that $s_{k+1} = h(s_k, a_k, w_k)$, where $w_k$ for $k \in \{0, 1, \dots \}$ is a sequence of independent and identically distributed (i.i.d.) random variables, and the
%time-invariant
immediate reward $r(s_k, a_k, w_k)$ received after taking action $a_k$ in state $s_k$.
A state-contingent decision policy is a mapping $\mu: \mathcal{S} \rightarrow \mathcal{A}$.
Given a discount factor $0 < q < 1$ and a policy $\mu$, the value function $V^\mu: \mathcal{S} \rightarrow \mathcal{R}$ is defined as
\begin{equation}
\label{eq:sum_MDP0}
    V^\mu (s) = \mathbb{E} \left [ \sum_{k = 0}^\infty q^k \cdot r(s_k, \mu(s_k), w_k) \bigg | s_0 = s \right ],
\end{equation}
where expectation is taken over $w_k$ for $k \geq 0$. Then, the optimal value function $V^*$ is defined by
\begin{equation}
\label{eq:sum_MDP}
    V^*(s) = \max_{\mu} V^\mu (s).
\end{equation}
%the goal in an MDP problem is to provide the agent with a sequence of state-contingent optimum actions $\{a_0^*(s_0), a_1^*(s_1), \dots\}$ that is the solution to
For a finite action space, any policy $\mu^*$ given by
\begin{equation}
\label{eq:optimal_action}
    \mu^*(s) = \argmax_{a \in \mathcal{A}} \ \mathbb{E} \big [ r(s, a, w) + q \cdot V^*(h(s, a, w)) \big ]
\end{equation}
is optimal in the sense that $V^*(s) = V^{\mu^*} (s)$,
%Let the value function $V^*(s)$ on the state space $\mathcal{S}$ be the optimum discounted sum of the rewards if the agent starts from state $s$, i.e., $V^*(s) = \max_{\{a_0, a_1, \dots\}} \mathbb{E} \left [ \sum_{k = 0}^\infty q^t \cdot r(S_k, A_k, w_k) \right ]$ given that $S_0 = s$.
%Equation \eqref{eq:sum_MDP} can be solved iteratively in an easy way resulting in the sequence of optimum actions if $V^*$ is known.
%The reason is that given $S_0 = s$, we have $\max_{\{a_0, a_1, \dots\}} \mathbb{E} \left [ \sum_{k = 0}^\infty q^k \cdot r(S_k, A_k, w_k) \right ] = \max_{a \in \mathcal{A}} \mathbb{E} \left [ r(s, a, w) + q \cdot V^*(h(s, a, w)) \right ]$,
which gives rise to the Bellman equation
\begin{equation}
\label{eq:bellman_equation}
    V^*(s) = \max_{a \in \mathcal{A}} \ \mathbb{E} \big [ r(s, a, w) + q \cdot V^*(h(s, a, w)) \big ] \quad \forall s \in \mathcal{S},
\end{equation}
where $w$ is a random variable with the same distribution as $w_k$ for some $k$.
Define the Bellman operator $\mathcal{T}$ as
\begin{equation}
\label{eq:Bellman_operator}
    (\mathcal{T}V)(s) = \max_{a \in \mathcal{A}} \ \mathbb{E} \big [ r(s, a, w) + q \cdot V(h(s, a, w)) \big ]
\end{equation}
%, which is discussed in detail in Section~\ref{sec:prob_contraction_expansion}.
Starting from an arbitrary $V_0$, the value iteration method constructs a sequence $\{V_0, V_1, V_2, \dots\}$ with $V_{t + 1} = \mathcal{T}(V_t)$ for $t \in \{0, 1, \dots\}$.  { It is well known that the Bellman operator is a contraction mapping, which guarantees convergence to $V^*$. The optimal value function $V^*$ is unknown in MDP and RL applications. The value function $V_t$ is a time-varying function and may never be exactly equal to $V^*$. Moreover, $V_t$ is rarely computed exactly and is subject to adversarial attacks. We will introduce multiple models of attack and analyze the corresponding notion of hitting time for each model to be able to study the convergence of $V_t$.}
%should be performed so that the maximum value of the argument function in Equation \eqref{eq:optimal_action} is approximated by maximizing the time-varying function $\mathbb{E} \big [ r(s, a, w) + q \cdot V_t(h(s, a, w)) \big ]$ with a desired precision.
%, which is studied in the literature.

\subsection{Related Work}

\subsubsection{Approximate Dynamic Programming}
 {
The field approximate dynamic programming encompasses a wide range of techniques that overcomes the curse of dimensionality in the computation of Bellman operator. The adversarial attack model studied in this paper is motivated by the following approaches:
\begin{enumerate}[label=\Roman*., wide, labelindent=10pt] %labelwidth=!, 
    \item \textbf{Approximation in computing expectation:}
    %Computing the expectation in {the} Bellman operator in Equation \eqref{eq:Bellman_operator} can be costly, given that it should be computed for every action in order to obtain the action that gives the highest reward.
    % There are different approaches to circumvent this issue , e.g., 
    There are different approaches to circumventing the costly computation of expectation in  \eqref{eq:Bellman_operator}, e.g., 
    a) assuming certainty equivalence by replacing stochastic quantities with deterministic ones {to arrive at} a deterministic optimization,
    % which can possibly degrade the accuracy significantly,
    b) using Monte Carlo tree search and adaptive simulation to determine which expectations associated {with} actions should be computed more accurately \cite{dimitri2017dynamic, chang2013simulation, coulom2006efficient, browne2012survey, fu2017markov}.
    Both of these approaches introduce some errors in {the} expectation. 
    % TODO There is another line of research on robust dynamic programming that addresses the uncertainty on the model parameters \cite{iyengar2005robust, nilim2005robust}.
    %; as a result, the contraction of Bellman operator may not be guaranteed in practice.
    \item \textbf{Approximation in maximization:}
    The maximization in {the} Bellman operator in \eqref{eq:Bellman_operator} can be over a large number of actions, possibly a continuous action space with {an} infinite number of actions.
    In addition to the discretization of the action space, nonlinear programming techniques are prone to errors especially when they are used in an online fashion.
    %    , such as gradient methods, Newton's method, {and} quadratic programming for deterministic problems with continuous action spaces, and stochastic programming for stochastic problems, are used to {handle} the large number of actions over which the maximization is performed \cite{dimitri2017dynamic}.
    %{These} methods 
    %Hence, the approximation error in maximization is another source of error.
    %possibly under-question the contraction in value iteration in practice.
    \item \textbf{Approximation of value function:}
    Due to the large number of states in many recent applications of Markov decision processes and reinforcement learning, %it may not be practical to have a memory associated with every state. Instead, 
    parametric feature-based approximation methods{, such as} neural network architectures, are used for value function representation \cite{dimitri2017dynamic, van1998learning, tsitsiklis1996feature, van2006performance, busoniu2010reinforcement}.
    The parameterization of the value function is another source of error in value iteration that can cause expansion in value iteration \cite{tsitsiklis1996feature, van1998learning}.
    % TODO This problem is resolved by taking further assumptions on parameterization methods to ensure contraction in {the} iterates of value iteration \cite{roy2006td, bertsekas2011approximate, de2000existence}.
    %However, under the next cause of error, adversarial attacks, the contraction in all iterates of value iteration may not hold anymore whose effect is studied in this {paper}.
    \item \textbf{Adversarial value iteration:}
    The emergence of cloud, edge, and fog computing means that large-scale MDP and RL problems will likely be solved by distributed servers \cite{satyanarayanan2017emergence, li2018learning, mach2017mobile}.
    This swift shift to edge reinforcement learning brings a host of new adversarial attack challenges that can be catastrophic in critical applications of autonomous vehicles and Internet of Things (IoT) in general \cite{isakov2019survey, ansari2020security, xiao2019edge}.
    %A natural example of such adversarial attacks is to contaminate the computation of value iteration so that the contraction of {the} Bellman operator would not hold in all iteraations 
    %In presence of distributed computing, the adversarial attacks can cause disturbances so that the disturbed Bellman operator is not a contraction mapping anymore.
\end{enumerate}} 
The first three causes {have been} studied extensively in the literature \cite{powell2009you}, {while there is no mathematical analysis of} adversarial attacks on {the computation of the value functions}.

\subsubsection{Reinforcement Learning in Time-varying Environment}
{Consider} a reinforcement learning framework in which the model is being learned or {there is} a time-varying environment whose state transition probabilities and rewards change over time~\cite{liu2018solution}. An example of a time-varying environment is the changing environment at which autonomous vehicles interact with each other, human drivers, and pedestrians. In the context of reinforcement learning and Markov decision processes, this gradual change is translated into time-varying reward functions and transition probabilities. The relevance of time-varying functions to MDP and RL problems presented above is one of the many problems that can be described by time-varying functions whose hitting time analysis is of interest.
{Other applications of} a time-varying framework, {such as bandit optimization, model predictive control, and empirical risk minimization, are discussed in } \cite{fenghitting}.
%Different other models of time-varying functions are studied in the rest of the article that find applications in blah blah that is discussed in \cite{fenghitting}.

\subsubsection{Scenario-based Approach for Optimization}
 {
Scenario-based approach for optimization~\cite{calafiore2005uncertain,campi2008exact,8299432} is concerned with decision making based on seen cases while having the ability to generalize to new situations. 
% ,
% A rigorous theory has been developed for convex  problems~\. 
In this context, a bound on the violation probability captures the generalization of time-invariant decisions.
% while the optimization model is indeed time-invariant.
The hitting time defined in this paper is related to the violation probability. 
Our work departs from this line of research in that we study a sequence of time-varying functions instead of a time-invariant function, which can potentially be corrupted by an adversary, and seeking to constantly adjusting our understanding of the optimal solution.
The hitting time captures the time-varying aspect in our setting.
% use the hitting time to capture the time-varying aspect in our setting
% captures the time-varying aspect in our setting;
% looking at a sequence of time-varying functions and are constantly adjusting our understanding of the optimal solution.
%  our setting involves time-varying functions that can be potentially be corrupted by an adversary.  % The notion of hitting time captures the time-varying aspect in our setting; (2) we exploit particular structure of the function and the structure of the time-varying
% operators in our analysis.
}
\subsubsection{Dynamical Systems}
 {
Our work is also related to asynchronous dynamical systems \cite{hassibi1999control}, which have been extensively studied in the literature.
Despite the mathematical resemblance, our work is different from this line of research since our focus is on analyzing the associated hitting times of different models and the dynamics considered in this work may not even be linear.
% which have been studied extensively in the literature.
}

\subsection{Contributions}
We propose a probabilistic model of adversarial attacks, in which {both} expansion up to a constant and contraction occur with certain probabilities in iterates of {the} value iteration {method. We then study the hitting} time of such stochastic time-varying value functions in Section \ref{problem_statement}.  {We develop an upper bound on the hitting time under a time-varying contraction mapping with additive noise and develop an upper bound on the distance between the fixed point and the value function.}
%{We refer to the} transformation in {this} probabilistic approach {as} a probabilistic contraction-expansion mapping, where the expansion can result from adversarial attacks disturbing the computation of value iteration on edge computing.

%The time-varying contraction mapping with additive noise is studied in Section \ref{sec:TV_contraction} and time-varying probabilistic contraction-expansion mapping with additive noise is studied in Section \ref{sec:TV_prob_mapping}.
% 
%Due to the emergence of} autonomous vehicles, the slim human-robot interaction {has been shifted} toward {a} robot-robot interaction {among} autonomous vehicles and a different human-robot interaction {among} vehicles, pedestrians, and human drivers as people adapt their actions to this new technology.

In the rest of {this paper}, different models of stochastic time variation for continuous and discrete functions are studied in Sections \ref{problem_statement} and \ref{sec:discrete}, respectively.
In particular, probabilistic contraction-expansion mappings are studied in Section \ref{sec:prob_contraction_expansion},
% time-varying contraction mappings with additive noise are studied in Section \ref{sec:TV_contraction},
time-varying probabilistic contraction-expansion mappings with additive noise are studied in Section \ref{sec:TV_prob_mapping},
time-varying continuous functions with additive noise are studied in Section \ref{section_optimization_TV_continuous_noise}, and improved bounds for convex functions with additive noise are studied in Section \ref{sec:improved_bounds_cts}.
% where we prove that convex functions are learned faster in general.
%, which is a nice property of the convex functions in addition to their many already known nice properties.
Time-varying discrete functions with additive noise are studied in Section \ref{sec:problem_statement2}, improved bounds for unimodal functions with additive noise are studied in Section \ref{sec:hitting_time_analysis_convex2}, and a time-varying linear model with additive noise with the notion of shape dominance are studied in Section \ref{sec:linear-model}.
%and \ref{sec:shape_dominante_model}.
 {We summarize the theorems and the associated assumptions as well as the hitting times definitions in Table \ref{tab:table1}.
Finally, the }simulation results are presented in Section \ref{sec:simulation} and the {paper} is concluded in Section~\ref{sec:conclusion} in which a discussion of opportunities for future work is presented as well.

%% file: sec/continuous.tex
\section{The Hitting Time Analysis for Continuous Functions}
\label{problem_statement}
In this section,  {three variants of stochastic time-varying models are studied and their hitting times are analyzed}.
In the first model, a probabilistic contraction-expansion mapping is {analyzed}, where
%that acts as a contraction mapping with a certain probability in consecutive rounds of value iteration; otherwise, acts as an expansion mapping up to a constant greater than one.
the classical Banach fixed-point theorem cannot be applied to this model {due to the probabilistic contraction-expansion nature of the problem}.
% In the second model, a time-varying but deterministic contraction mapping with additive noise is studied.
In the  {second} model, a time-varying probabilistic contraction-expansion mapping with additive noise is {investigated}.
 {The above two models} are applicable to both continuous and discrete functions.
In the last model, an unknown time-varying continuous function is observed with additive noise whose estimated function changes over time.
%The differences between this model and the one studied in Section \ref{sec:problem_statement2} are that the target function is also changing over time in this new model and the target function is continuous, which brings more subtleties into the proof.

 {To motivate the three stochastic time-varying models, we revisit the motivating example in the previous section, where a sequence of value functions $V_0, V_1, \dots $ is generated by the Bellman operator $\mathcal{T}$ defined in \eqref{eq:Bellman_operator}. 
Note that the theoretical proof of convergence {behind the} value iteration {method} depends heavily on the contraction mapping parameter $q$ and the fact that $d \big ( \mathcal{T}(V_{t + 1}), \mathcal{T}(V_t) \big ) \leq q \cdot d \big ( V_{t + 1}, V_t \big )$ deterministically{, where $d(\cdot, \cdot)$ is a translation-invariant distance function induced by a norm}.
However, in an online implementation of the value iteration with large state or action spaces, the actual calculation in practice may result in the value iteration {method} not to satisfy the contraction condition $d \big ( \mathcal{T}(V_{t + 1}), \mathcal{T}(V_t) \big ) \leq q \cdot d \big ( V_{t + 1}, V_t \big )$ in some iterations.
Instead, the distance may expand up to a factor greater than one in some iterations of the value iteration, i.e., $d \big ( \mathcal{T}(V_{t + 1}), \mathcal{T}(V_t) \big ) \leq Q \cdot d \big ( V_{t + 1}, V_t \big )$, where $Q \geq 1$. {In this problem,} the Bellman contraction mapping in value iteration may not be fixed anymore {and could} change over time.
Hence, instead of applying the same transformation $\mathcal{T}$ in value iteration, a time-varying transformation $\mathcal{T}_t$ for $t \in \{0, 1, \dots\}$ may be applied to value iteration. Section~\ref{sec:prob_contraction_expansion} formalizes this observation. }

\subsection{Probabilistic Contraction-Expansion Mapping}
\label{sec:prob_contraction_expansion}
Let $\left (X, \| \cdot \| \right )$ be a non-empty complete normed vector (linear) space, known as a Banach space, over the field $\mathbb{R}$ of real scalars, where $X$ is a vector space, e.g., a function space, together with a norm $\| \cdot \|$.
The norm induces a translation invariant distance function, called canonical induced metric, as $d(f, g) = \|f - g\|$.
Let $\| f \| = \langle f, f \rangle^{1/2}$, where the inner product of $f, g \in X$ in general is defined by $\langle f, g \rangle = \int f(x) g(x) dx$.
Consider a contraction mapping $\mathcal{T}: X \rightarrow X$ with the property that for all $f, g \in X$, there exists {a scalar $q \in [0, 1)$} such that
\begin{equation}
\label{contraction_deterministic}
d \big (\mathcal{T}(f), \mathcal{T}(g) \big ) \leq q \cdot d(f, g).
\end{equation}
{In light of} the Banach-Caccioppoli fixed-point theorem, {this} contraction mapping has its own unique fixed point, i.e., there exists $f^* \in X$ such that $\mathcal{T}(f^*) = f^*$.
Furthermore, starting with an arbitrary function $f^0 \in X$, the sequence $\{f^n\}$ with $f^n = \mathcal{T}(f^{n - 1})$ for $n \geq 1$ converges to $f^*$; in other words, $f^n \rightarrow f^*$, where $d \big ( f^*, f^n \big ) \leq \frac{q^n}{1 - q} \cdot d(f^1, f^0)$.
% If the goal is to optimize the function {$f^*$}, the optimization of the {functions in the sequence $\{f^n\}$} may not be reliable up until a hitting time in which the function sequence is close enough to $f^*$ and does not change dramatically.
% In the current and next subsections, different variants of mappings are studied and their hitting times are analyzed.
Note that in all iterations of the above value iteration,  {the mapping $\mathcal{T}$ operates as a contraction mapping according to \eqref{contraction_deterministic} with probability one.}
% the same contraction mapping $\mathcal{T}$ is applied to the function sequence that operates as a contraction mapping 
However, in the rest of this subsection, we consider a probabilistic version of the Banach fixed-point theorem, where the mapping either contracts or expands the distance between any two points in a probabilistic manner.

Consider the time-varying function $f_t \in X$ for $t \in \{ 0, 1, 2, \dots \}$ evolving over time according to
\begin{equation}
\label{time_varying_map_1}
    f_{t + 1} = \overline{\mathcal{T}}(f_t), \quad t \in \{ 0, 1, 2, \dots \},
\end{equation}
where $\overline{\mathcal{T}}$ is a probabilistic contraction-expansion mapping such that
%if applied independently to  performs as a contraction map associated with $q$ with probability (w.p.) $p$ and performs as an expansion map up to a bounded constant $Q \geq 1$ otherwise such that $q^2 \cdot p + Q^2 \cdot (1 - p) < 1$, i.e., for all $f, g \in X$, we have
\begin{equation}
\label{probabilistic_contraction_expansion}
%d \big ( f_{t + 2}, f_{t + 1} \big ) =
d \big (\overline{\mathcal{T}}(f_{t+1}), \overline{\mathcal{T}}(f_t) \big ) \leq
\begin{cases}
      q \cdot d(f_{t + 1}, f_t) & \text{w.p.} \quad p\\
      Q \cdot d(f_{t + 1}, f_t) & \text{otherwise}
\end{cases}, \ \ \forall t \in \mathbb{N}_0 %\qquad \forall t \in \{ 0, 1, 2, \dots \}
\end{equation}
{for some constants $q \in [0, 1)$, $Q \geq 1$, and $p \in (0, 1]$, where w.p. stands for ``with probability'' and $\mathbb{N}_0$ is natural numbers with zero.}
%where $q^2 \cdot p + Q^2 \cdot (1 - p) < 1$ and the abbreviation with probability (w.p.) is used.
 {{The expansion in \eqref{probabilistic_contraction_expansion} is caused by an adversary in an attempt to move the function sequence away from the fixed point.}
The contraction or expansion of $\overline{\mathcal{T}}$ is independent over time and $f^*$ is a fixed point of the mapping if $\overline{\mathcal{T}}(f^*) = f^*$.}
The shape of the function $f_t$ changes over time, but there can be a time, called hitting time $T$,  {at which $f_T$ reaches {a neighborhood of $f^*$,} as formally defined below.}

\begin{definition}
\label{def:hitting_time_probabilistic}
{Given $\epsilon > 0$ and $a \in (0, 1]$,} the hitting time $T(\epsilon, a)$ for the stochastic function sequence introduced in \eqref{time_varying_map_1} %with $f^* = \lim_{t \rightarrow \infty} f_t$
is defined as
\begin{equation}
    \label{hitting_time_NTVT}
    T(\epsilon, a) = \min \big \{ T : \mathbb{P} \left \{ d \big ( f_t, f^* \big ) < \epsilon \right \} \geq 1 - a, \ \forall t \geq T \big \},
\end{equation}
where $f^*$ is a fixed point {whose existence and uniqueness is proven in Theorem \ref{thm:probabilistic_banach} and $\mathbb{P}\{\cdot\}$ takes the probability of the input event}.
\end{definition}
 {
As a result, the complexity of optimizing {the} functions $f_t$ for $t < T$ can be irrelevant to the optimization complexity of {the} functions $f_t$ for $t \geq T$.
Consequently, the hitting time $T$ together with {the} optimization complexity of {any} function $f_t$ for $t \geq T$ captures the complexity of optimizing the time-varying sequence of functions $\{ f_t \}$.}
% {The focus of the rest of this section is to find an upper bound on} the hitting time for the function sequence $\{f_t\}$.
In the following theorem, the limiting behavior of the function sequence $\{ f_t \}$ is studied and an upper bound {on the} hitting time {is derived}.

%In order to study the hitting time of this model, we first need to analyze the limiting behavior of the function sequence, which is presented as the probabilistic version of the Banach fixed-point theorem in the following theorem.

\begin{theorem}
\label{thm:probabilistic_banach}
\textbf{Probabilistic Banach Fixed-Point Theorem.}
Let $\left (X, \| \cdot \| \right )$ be a non-empty complete normed vector space with a probabilistic contraction-expansion mapping $\overline{\mathcal{T}}: X \rightarrow X$ defined in \eqref{probabilistic_contraction_expansion} {such that $q^2 \cdot p + Q^2 \cdot (1 - p) < 1$}.
Starting with an arbitrary element $f_0 \in X$, the sequence $\{ f_t \}$ defined in \eqref{time_varying_map_1} converges to an element $f^* \in X$ with {an associated confidence level $1 - a$}, where $f^*$ is a unique fixed point for {the} mapping $\overline{\mathcal{T}}$.
%Then $\overline{\mathcal{T}}$ admits a unique fixed-point $f^* \in X$ such that $\overline{\mathcal{T}} \big ( f^* \big ) = f^*$ with high probability.
Furthermore,  {for every $0 < L < \frac{\epsilon}{d \left ( f_1, f_0 \right )}$,}
the hitting time {$T(\epsilon, a)$ satisfies the inequality}
%in Definition \ref{def:hitting_time_probabilistic} is upper bounded as
\begin{equation}
\label{theorem_equation_prob_Banach}
    \begin{aligned}
    T(\epsilon, a) \hspace{-0.5mm} \leq \hspace{-0.4mm} & \max \hspace{-0.5mm} \left\{ \hspace{-0.5mm} \frac{\ln \hspace{-0.5mm} \left ( \hspace{-1mm} \frac{a \cdot L^2 \hspace{-0.5mm}\cdot\hspace{-0.5mm} \big (1 - q \cdot p - Q \cdot (1 - p) \big ) \hspace{-0.5mm}\cdot\hspace{-0.5mm} \big ( 1 - q^2 \cdot p - Q^2 \cdot (1 - p) \big )}{ 1 + q \cdot p + Q \cdot (1 - p)} \hspace{-1mm} \right )}{\ln \big ( q^2 \cdot p + Q^2 \cdot (1 - p) \big )}\hspace{-0.5mm},\hspace{-0.5mm} \frac{\ln\hspace{-0.5mm} \left(\hspace{-0.5mm} \Big (\hspace{-0.5mm} \frac{\epsilon}{d \left ( f_1, f_0 \right )} \hspace{-0.5mm}-\hspace{-0.5mm} L \Big ) \hspace{-0.5mm}\cdot\hspace{-0.5mm} \Big (\hspace{-0.5mm} 1 - q \cdot p - Q \cdot (1 - p) \hspace{-0.5mm}\Big )\hspace{-0.5mm} \right )\hspace{-0.5mm}}{\ln \big (q \cdot p + Q \cdot (1 - p) \big )} \hspace{-0.5mm}\right\}\hspace{-0.5mm}.\hspace{-1mm}
    \end{aligned}
\end{equation}
\end{theorem}

\begin{proof}
In order to find an upper bound {on} the hitting time $T(\epsilon, a)$ defined in Definition \ref{def:hitting_time_probabilistic}, we first need to study  {the convergence behavior of} the function sequence $\{f_t\}$ in \eqref{time_varying_map_1} under the probabilistic contraction-expansion mapping $\overline{\mathcal{T}}$.
To this end, we prove that this function sequence is a Cauchy sequence with high probability.
{Given arbitrary} integer values $n$ and $m$ such that $n > m$, {one can write}
\begin{equation}
\label{distance_probabilistic}
\begin{aligned}
    d \big ( f_n, f_m \big ) = d \big ( \overline{\mathcal{T}}^n(f_0), \overline{\mathcal{T}}^m(f_0) \big ) 
    & \overset{(a)}{\leq} \hspace{-1.2mm} \sum_{i = 1}^{n - m}    d \big ( \overline{\mathcal{T}}^{n - i + 1}  (f_0), \overline{\mathcal{T}}^{n - i}  (f_0) \big )   
    =    \sum_{i = 1}^{n - m}    d \big ( \overline{\mathcal{T}}^{n - i}  (f_1), \overline{\mathcal{T}}^{n - i}  (f_0) \big ) \\
    & \overset{(b)}{\leq} \sum_{i = 1}^{n - m}  {\left(\prod_{j = 1}^{n - i} B_j\right)} \cdot d \big (f_1, f_0 \big ) =
    d \big (f_1, f_0 \big ) \cdot \sum_{i = 1}^{n - m} \prod_{j = 1}^{n - i} B_j,
\end{aligned}
\end{equation}
where triangular inequality is applied $n - m - 1$ times in $(a)$ and the independent and identically distributed random variables $B_j$ for $j \in \{1, 2, \dots, n - 1\}$ used in $(b)$ have the distribution
\begin{equation}
\label{B_distribution}
    B_j =
    \begin{cases}
      q & \text{w.p.} \quad p \\
      Q & \text{otherwise}
    \end{cases}.
\end{equation}
Next, we study the mean and variance of {the} random variable $S_{n, m} = \sum_{i = 1}^{n - m} \prod_{j = 1}^{n - i} B_j$ in \eqref{distance_probabilistic}.
Using the independence of $B_j$ for $j \in \{1, 2, \dots, n - 1 \}$, the mean can be  {upper-bounded} as
\begin{equation}
    \label{mean_prob}
    \begin{aligned}
    \mathbb{E} [S_{n, m}] = \mathbb{E} \left [ \sum_{i = 1}^{n - m} \prod_{j = 1}^{n - i} B_j \right ]
    = \sum_{i = 1}^{n - m} \prod_{j = 1}^{n - i} \mathbb{E} \left [ B_j \right ]
     = \sum_{i = 1}^{n - m} \big (q \cdot p + Q \cdot (1 - p) \big )^{n - i}  \leq \frac{\big (q \cdot p + Q \cdot (1 - p) \big )^m}{1 - q \cdot p - Q \cdot (1 - p) }.
    \end{aligned}
\end{equation}
On the other hand, $\Var(S_{n, m}) \leq \mathbb{E} \left [ S_{n, m}^2 \right ]$, {where $\Var(\cdot)$ takes the variance of the input random variable}, and the second moment of $S_{n, m}$ {will be}  {upper-bounded} {next}.
Note that
\begin{equation}
    \label{second_moment}
    \begin{aligned}
    S_{n, m} = B_1 \cdot B_2 \cdots B_m \cdot \big ( 1 + & B_{m + 1} + B_{m + 1} \cdot B_{m + 2} + \dots + B_{m + 1} \cdots B_{n - 1} \big ).
    \end{aligned}
\end{equation}
Let $\bar{S}_{n, m} = 1 + B_{m + 1} + B_{m + 1} \cdot B_{m + 2} + \dots + B_{m + 1} \cdots B_{n - 1}$, where  {$\bar{S}_{n, m}$ is a random variable independent of $B_j$ for $j \in \{1, 2, \dots, m \}$, and $\bar{S} = \lim_{n \rightarrow \infty} \bar{S}_{n, m}$.
We leave out the subscript $m$ since the limits $\lim_{n \rightarrow \infty} \bar{S}_{n, m}$ and $\lim_{n \rightarrow \infty} \bar{S}_{n, m^\prime}$ are identically distributed for all $m,m^\prime \geq 0$.
This is because $\bar{S}$ is an infinite sum and $\{B_j\}$ are i.i.d. random variables.
}
% The random variable $\bar{S}$ is independent of $m$ because it is an infinite sum and $\{B_j\}$ are i.i.d. random variables.
% (the dependence on $m$ is dropped after taking the limit as $\bar{S}$ is an  and ).
Since $\mathbb{E}[B_j] > 0$ for $j \geq 1$, {we have} $\mathbb{E} [ \bar{S}_{n, m}^2 ] \leq \mathbb{E} [\bar{S}^2]$; hence, {it follows from \eqref{second_moment} that}
\begin{equation}
    \label{S_n_m_second_moment}
    \begin{aligned}
    \mathbb{E} \left [ S_{n, m}^2 \right ]
    & = \mathbb{E} \left [ B_1^2 \right ] \cdots \mathbb{E} \left [ B_m^2 \right ] \cdot \mathbb{E} \left [ \bar{S}_{n, m}^2 \right ]  \leq \mathbb{E} \left [ B_1^2 \right ] \cdots \mathbb{E} \left [ B_m^2 \right ] \cdot \mathbb{E} \left [ \bar{S}^2 \right ].
    \end{aligned}
\end{equation}
In order to find an upper bound on $\mathbb{E} \left [ \bar{S}^2 \right ]$, we have
\begin{equation}
    \label{S_bar}
    \begin{aligned}
    \bar{S} = 1 + & B_{m + 1} \cdot (1 + B_{m + 2} + B_{m + 2} \cdot B_{m + 3} +  B_{m + 2} \cdot B_{m + 3} \cdot B_{m + 4} + \dots) = 1 + B_{m + 1} \cdot \Tilde{S},
    \end{aligned}
\end{equation}
where $\Tilde{S}$ is independent of $B_{m + 1}$, and {the} random variables $\bar{S}$ and $\Tilde{S}$ are identically distributed but not independent of each other.
{By} taking expectation on both sides of $\bar{S}^2 = (1 + B_{m + 1} \cdot \Tilde{S})^2$, {and} using the independence of $\Tilde{S}$ and $B_{m + 1}$ and the fact that $\mathbb{E} \big [ \bar{S}^2 \big ] = \mathbb{E} \big [ \Tilde{S}^2 \big ]$, {one can obtain}

\begin{equation}
    \label{second_moment_s_barr}
    \begin{aligned}
    & \mathbb{E} \left [ \bar{S}^2 \right ] = 1 + \mathbb{E} \left [ B_{m + 1}^2 \right ] \cdot \mathbb{E} \left [ \Tilde{S}^2 \right ] + 2 \mathbb{E} \left [ B_{m + 1} \right ] \cdot \mathbb{E} \left [ \Tilde{S} \right ] \Longrightarrow  \mathbb{E} \left [ \bar{S}^2 \right ] = \frac{1 + 2 \mathbb{E} \left [ B_{m + 1} \right ] \cdot \mathbb{E} \left [ \Tilde{S} \right ]}{1 - \mathbb{E} \left [ B_{m + 1}^2 \right ]}.
    \end{aligned}
\end{equation}
In the same way as finding the mean of $S_{n, m}$ in \eqref{mean_prob}, it is derived that $\mathbb{E} \left [ \Tilde{S} \right ] = \frac{1}{1 - q \cdot p - Q \cdot (1 - p)}$; furthermore, $\mathbb{E} \left [ B_{m + 1} \right ] = q \cdot p + Q \cdot (1 - p)$ and $\mathbb{E} \left [ B_{m + 1}^2 \right ] = q^2 \cdot p + Q^2 \cdot (1 - p)$.
As a result, if $q^2 \cdot p + Q^2 \cdot (1 - p) < 1$, Equation \eqref{second_moment_s_barr} results in
\begin{equation}
    \label{s_bar_second_momentt}
    \mathbb{E} \left [ \bar{S}^2 \right ] = \frac{1 + q \cdot p + Q \cdot (1 - p)}{ \big (1    -    q \cdot p    -    Q \cdot (1    -    p) \big )    \cdot    \big ( 1    -    q^2 \cdot p    -    Q^2 \cdot (1    -    p) \big )}.
\end{equation}
Using Equation \eqref{S_n_m_second_moment}, we have
\begin{equation}
    \label{variance_prob}
    \begin{aligned}
     \Var(S_{n, m})\leq \mathbb{E} \left [ S_{n, m}^2 \right ]
    \hspace{-0.5mm}\leq\hspace{-0.5mm} \big ( q^2 \cdot p + Q^2 \cdot (1 - p) \big )^m \hspace{-0.5mm}\times\hspace{-0.5mm}  \frac{1 + q \cdot p + Q \cdot (1 - p)}{ \big (1 - q \cdot p - Q \cdot (1 - p) \big ) \cdot \big ( 1 - q^2 \cdot p - Q^2 \cdot (1 - p) \big )}.
    \end{aligned}
\end{equation}
%%%%%%%Note that $\left ( \mathbb{E} \left [ S_{n, m}^2 \right ] \right )^2$ can be subtracted from the second moment of $S_{n, m}$ in the above inequality to obtain a tighter upper bound for variance of $S_{n, m}$, but that complicates the next equations, so we ignore that term.
%The steps taken above are summarized below.
%It is derived in Equation \eqref{distance_probabilistic} for any $n > m$
{So far, it is shown}
that $d \big ( \overline{\mathcal{T}}^n(f_0), \overline{\mathcal{T}}^m(f_0) \big ) \leq S_{n, m} \cdot d \big (f_1, f_0 \big )$, where $S_{n, m}$ is a random variable with {its} mean and variance  {upper-bounded in} \eqref{mean_prob} and \eqref{variance_prob}, respectively.
Using Chebyshev's inequality, for any $L > 0$, we have
\begin{equation}\label{chebyshev_inequality}
    \begin{aligned}
    & \mathbb{P} \left \{ | S_{n, m}  - \mathbb{E} [ S_{n, m} ] | \leq L \right \} \geq 1 - \frac{\Var(S_{n, m})}{L^2} \Longrightarrow \\
    & \mathbb{P} \left \{ S_{n, m} \leq \frac{\big (q \cdot p + Q \cdot (1 - p) \big )^m}{1 - q \cdot p - Q \cdot (1 - p) } + L \right \} \hspace{-1mm}\geq\hspace{-0.5mm}  1    -    \frac{ \big ( q^2 \cdot p + Q^2 \cdot (1 - p) \big )^m \cdot \big ( 1 + q \cdot p + Q \cdot (1 - p) \big )}{L^2    \cdot    \big (1    -    q    \cdot    p    -    Q    \cdot    (1    -    p) \big )    \cdot    \big ( 1    -    q^2    \cdot    p    -    Q^2    \cdot    (1    -    p) \big )}.
    \end{aligned}
\end{equation}
As a result, for any $\epsilon > 0$ and {$a \in (0, 1]$}, we have  {$d(f_n,f_m) = d \big ( \overline{\mathcal{T}}^n(f_0), \overline{\mathcal{T}}^m(f_0) \big ) \leq \epsilon$} with {the} confidence level $1 - a$ if $m$ satisfies the two inequalities
\begin{subequations}
\label{two_inequality_for_m}
\begin{eqnarray}
%\begin{aligned}
\label{two_inequality_for_m_a}
  \dfrac{ \big ( q^2 \cdot p + Q^2 \cdot (1 - p) \big )^m \cdot \big ( 1 + q \cdot p + Q \cdot (1 - p) \big )}{L^2 \cdot \big (1 - q \cdot p - Q \cdot (1 - p) \big ) \cdot \big ( 1 - q^2 \cdot p - Q^2 \cdot (1 - p) \big )} &\leq a \\
 \label{two_inequality_for_m_b}
  \left ( \dfrac{\big (q \cdot p + Q \cdot (1 - p) \big )^m}{1 - q \cdot p - Q \cdot (1 - p) } + L \right ) \cdot d \big ( f_1, f_0 \big ) &\leq \epsilon.
%\end{aligned}
\end{eqnarray}
\end{subequations}
{Assume} that $d \big ( f_1, f_0 \big ) \neq 0$; otherwise, $f_0$ is a fixed point by definition.
Hence, {for} $0 < L < \frac{\epsilon}{d \left ( f_1, f_0 \right )}$, {if $q \cdot p + Q \cdot (1 - p) < 1$ and $q^2 \cdot p + Q^2 \cdot (1 - p) < 1$,} then the two inequalities in \eqref{two_inequality_for_m_a} and \eqref{two_inequality_for_m_b} are satisfied {when}
\begin{equation}
    \label{N_epsilon_value}
    \begin{aligned}
    m \geq \max \hspace{-0.5mm} \left\{ \hspace{-0.5mm} \frac{\ln \hspace{-0.5mm} \left ( \hspace{-1mm} \frac{a \cdot L^2 \hspace{-0.5mm}\cdot\hspace{-0.5mm} \big (1 - q \cdot p - Q \cdot (1 - p) \big ) \hspace{-0.5mm}\cdot\hspace{-0.5mm} \big ( 1 - q^2 \cdot p - Q^2 \cdot (1 - p) \big )}{ 1 + q \cdot p + Q \cdot (1 - p)} \hspace{-1mm} \right )}{\ln \big ( q^2 \cdot p + Q^2 \cdot (1 - p) \big )}\hspace{-0.5mm},\hspace{-0.5mm} \frac{\ln\hspace{-0.5mm} \left(\hspace{-0.5mm} \Big (\hspace{-0.5mm} \frac{\epsilon}{d \left ( f_1, f_0 \right )} \hspace{-0.5mm}-\hspace{-0.5mm} L \Big ) \hspace{-0.5mm}\cdot\hspace{-0.5mm} \Big (\hspace{-0.5mm} 1 - q \cdot p - Q \cdot (1 - p) \hspace{-0.5mm}\Big )\hspace{-0.5mm} \right )\hspace{-0.5mm}}{\ln \big (q \cdot p + Q \cdot (1 - p) \big )} \hspace{-0.5mm}\right\}.
    \end{aligned}
\end{equation}
%%%%{Although the above inequality holds for any $0 < L < \frac{\epsilon}{d \left ( f_1, f_0 \right )}$, the right-hand side of the inequality can be minimized by tuning the value of $L$.}
%The parameter $L$ can be anything in the interval $\Big (0, \frac{\epsilon_p}{d \left ( \bar{f}_1, \bar{f}_0 \right )} \Big )$.
 {Now, for every $\epsilon>0$ and $a\in (0,1]$, let $N_\epsilon$ be the constant on the right-hand side of \eqref{N_epsilon_value}. 
Then, with probability $1-a$, it holds that $\lim_{n \rightarrow \infty} d(f_n,f_{N_\epsilon})\leq \lim_{n \rightarrow \infty} S_{n,N_\epsilon} \cdot d \big (f_1, f_0 \big ) \leq \epsilon$. 
For all $n > m > N_\epsilon$, since $\{B_j\}$ are nonnegative, it holds that $S_{n,m}=B_{1} \cdot B_{2} \cdots B_{N_{\epsilon}} \cdot\left(B_{N_{\epsilon}+1} \cdots B_{m}+\cdots+B_{N_{\epsilon}+1} \cdots B_{n-1}\right)\leq B_{1} \cdot B_{2} \cdots B_{N_{\epsilon}} \cdot (1+B_{N_{\epsilon}+1}+B_{N_{\epsilon}+1}B_{N_{\epsilon}+2}+\dots) = \lim_{n \rightarrow \infty} S_{n,N_\epsilon}$, which implies $d(f_n,f_m)\leq S_{n,m}\cdot d \big (f_1, f_0 \big ) \leq \epsilon$ as long as $\lim_{n \rightarrow \infty} S_{n,N_\epsilon} \cdot d \big (f_1, f_0 \big ) \leq \epsilon$.
To conclude, the sequence $\{ f_t \}$ is a Cauchy sequence with probability $1-a$.}
{Since} the {vector} space {$X$} is complete, the sequence $\{ f_t \}$ converges to an element $f^*$ in the space with high probability.
Moreover, $f^*$ is a fixed point of the mapping $\overline{\mathcal{T}}$ since with high probability we have
\begin{equation}
    \label{fixed_point_probabilistic}
    \overline{\mathcal{T}} (f^*)
    = \overline{\mathcal{T}} ( \lim_{t \rightarrow \infty } f_t)
    \overset{(a)}{=} \lim_{t \rightarrow \infty } \overline{\mathcal{T}} ( f_t) \\
    = \lim_{t \rightarrow \infty } f_{t + 1} = f^*,
\end{equation}
where $(a)$ is true as the mapping $\overline{\mathcal{T}}$ is continuous due to \eqref{probabilistic_contraction_expansion}, which justifies bringing the limit outside the operator $\overline{\mathcal{T}}$.
Lastly, there cannot be more than one fixed point for {the} mapping $\overline{\mathcal{T}}$, which can be proved by contradiction.
Considering any pair of distinct fixed points $f_1^*$ and $f_2^*$, we have $d \big ( \overline{\mathcal{T}}(f_1^*), \overline{\mathcal{T}}(f_2^*) \big ) = d \big ( f_1^*, f_2^* \big )$ with probability {1}, which contradicts the fact that the distance between the mapped points contracts with a factor $q < 1$ with probability $p > 0$.

% then for all , it holds that $d \big ( \overline{\mathcal{T}}^n(f_0), \overline{\mathcal{T}}^m(f_0) \big ) \leq S_{n, m} \cdot d \big (f_1, f_0 \big )\leq \epsilon$}
% {for a constant} $N_\epsilon$ {that} is greater than the term on the right-hand side of Equation \eqref{N_epsilon_value}, for all $n > m > N_\epsilon$ {it holds that} $d \big ( \overline{\mathcal{T}}^n(f_0), \overline{\mathcal{T}}^m(f_0) \big ) \leq S_{n, m} \cdot d \big (f_1, f_0 \big ) = B_1 \cdot B_2 \cdots B_{N_\epsilon} \cdot \big ( B_{N_\epsilon + 1} \cdots B_{m} + \dots + B_{N_\epsilon + 1} \cdots B_{n - 1} \big ) \cdot d \big (f_1, f_0 \big ) \leq \epsilon$ with high probability since $\big ( B_{N_\epsilon + 1} \cdots B_{m} + \dots + B_{N_\epsilon + 1} \cdots B_{n - 1} \big )$ is deterministically increasing as $n$ goes to infinity and it is already proven above that $S_{n, m} \cdot d \big (f_1, f_0 \big ) \leq \epsilon$ with high probability as $n$ goes to infinity.

In {this} proof, both $q \cdot p + Q \cdot (1 - p) < 1$ and $q^2 \cdot p + Q^2 \cdot (1 - p) < 1$ must be satisfied {to ensure that Equations \eqref{two_inequality_for_m_a} and \eqref{two_inequality_for_m_b} hold for a large enough $m$}.
{However,} $q^2 \cdot p + Q^2 \cdot (1 - p) < 1$ {implies} $q \cdot p + Q \cdot (1 - p) < 1$ {since one can write}
\begin{equation}\label{eq:implication_of_assumption}
    \begin{aligned}
         (1 - p) \cdot (Q^2 - 2Q + 1) \geq 0  
        \Longrightarrow\quad& Q^2\quad \cdot (1 - p) - 2Q \cdot (1 - p) + 1 - p \geq 0  \\
        \overset{(a)}{\Longrightarrow}\quad& Q^2    \cdot    (1    -    p)^2    -    2Q \cdot (1 - p) + 1 \geq p \cdot \big (1 - (1 - p) \cdot Q^2 \big )  \\
        \overset{(b)}{\Longrightarrow}\quad& 1 - Q \cdot (1 - p) \geq p \cdot \sqrt{\frac{1 - Q^2 \cdot (1 - p)}{p}}  \\
        \overset{(c)}{\Longrightarrow}\quad& q \cdot p + Q \cdot (1 - p) < 1,
    \end{aligned}
\end{equation}
where $p - p \cdot (1 - p) \cdot Q^2$ is added on both sides of inequality in $(a)$, the square root is taken from both sides in $(b)$, and $q^2 \cdot p + Q^2 \cdot (1 - p) < 1$ is used in $(c)$ to draw the claimed conclusion.% \qquad \Halmos
%\begin{flushright}
%\Halmos
%\end{flushright}
\end{proof}

{Theorem \ref{thm:probabilistic_banach} states that if contraction of an operator in the iterates of the value iteration is compromised by an adversary via expansions in the iterates of value iteration, the value function sequence can still converge to the fixed point of the operator with high probability.
 {The standard Banach fixed-point theorem is a special case of Theorem \ref{thm:probabilistic_banach} by setting $p = 1$ and $L=0$.}
The analysis in the proof of this theorem suggests that the compromised operator being contractive on expectation is not enough for the convergence of the value function sequence with high probability since the introduced randomness to the operator by the adversary can lead to high variance in the elements of the value function sequence.
 {Hence, the additional assumption $q^2 \cdot p + Q^2 \cdot (1 - p) < 1$
% on the contraction and expansion of the operator 
is required to bound such a variance rooted from the expansion caused by the adversary.}
Furthermore, this theorem provides an upper bound on the number of rounds for value iteration to defeat the effect of the adversary that attempts to move the value function sequence away from the fixed point.
 {If the adversary is not modeled, the user who expects a normal scenario may perform {fewer iterations} of the value iteration. 
This can lead to a highly inaccurate estimate of the fixed point in the presence of an adversary.}
% that is used to maximize the expected reward in a possibly critical application.
}

\begin{remark}
 {The parameter $L\in \left(0, \frac{\epsilon}{d \left ( f_1, f_0 \right )}\right)$ serves as an auxiliary parameter used in \eqref{chebyshev_inequality}.
We observe that the first term in the upper bound \eqref{theorem_equation_prob_Banach} is decreasing with respect to $L$ and the second term is increasing with respect to $L$.
By minimizing the bound \eqref{theorem_equation_prob_Banach} over $L$,
%of Theorem \ref{thm:probabilistic_banach}
% over the {auxiliary} parameter, 
we have that  {$T(\epsilon,a)$ has the order $\mathcal{O}\left(\frac{d \left ( f_1, f_0 \right )}{\epsilon}\right)$}.}
\end{remark}

\subsection{Time-Varying Probabilistic Contraction-Expansion Mapping with Additive Noise}
\label{sec:TV_prob_mapping}

Let $(X, \| \cdot \| )$ be the same complete normed vector space as in Section \ref{sec:prob_contraction_expansion}.
Consider time-varying probabilistic contraction-expansion  {mappings} $\overline{\mathcal{T}}_t(\cdot): X \rightarrow X$ for $t \in \{ 0, 1, 2, \dots \}$ with parameters $p_t, q_t,$ and $Q_t$,  {i.e.,}
\begin{equation}
%d \big ( f_{t + 2}, f_{t + 1} \big ) =
 {d \big (\overline{\mathcal{T}}_t(f), \overline{\mathcal{T}}_t(g) \big ) \leq
\begin{cases}
      q_t \cdot d(f, g) & \text{w.p.} \quad p_t\\
      Q_t \cdot d(f, g) & \text{otherwise}
\end{cases}, \ \ \forall t \in \mathbb{N}_0.}
\end{equation}
\noindent  {By Theorem \ref{thm:probabilistic_banach}, starting with an arbitrary function $f^0 \in X$,} the sequence $\{f^n\}$ with $f^n = \overline{\mathcal{T}}_t(f^{n - 1})$ for $n \geq 1$, where the same probabilistic contraction-expansion mapping $\overline{\mathcal{T}}_t$ is applied repeatedly, converges to $f_t^*$ with high probability.
 {
\begin{assumption}\label{assump:consecutive_fixed_point}
The fixed points of every two consecutive  {mappings} are at most $\epsilon_f > 0$ away from each other, i.e., $d \big ( f_t^*, f_{t - 1}^* \big ) \leq \epsilon_f$ for all $t \in \{ 1, 2, 3, \dots\}$.
\end{assumption}}
 {It is worth mention that, even under Assumption \ref{assump:consecutive_fixed_point}, there can be non-consecutive mappings $\overline{\mathcal{T}}_t$ and $\overline{\mathcal{T}}_{t'}$ whose fixed points {are arbitrarily far away from each other.}}
% the maximum distance $\max_{t,t^\prime \geq 0}d(f_t^\star,f_{t^\prime}^\star)$ can be arbitrarily large.}
% {Assume} that the fixed points of any two consecutive transformations are at most $\epsilon_f > 0$ away from each other, i.e., $d \big ( f_t^*, f_{t - 1}^* \big ) \leq \epsilon_f$ for all $t \in \{ 1, 2, 3, \dots\}$.
% Nevertheless, }
%given that the above assumption is satisfied.
Note that in all iterations of the probabilistic value iteration, the same probabilistic contraction-expansion mapping $\overline{\mathcal{T}}_t$ is applied to the function sequence $\{f^n\}$.
However, in the {remainder} of this subsection, we consider a time-varying and noisy version of the probabilistic Banach fixed-point theorem, where the {underlying}  {mapping} changes over time and noise functions are added to the outcome of the mapping in each iteration.

Consider the time-varying function $f_t \in X$ for $t \in \{ 0, 1, 2, \dots \}$ evolving over time according to
\begin{equation}
\label{time_varying_map_2}
    f_{t + 1} = \widetilde{\mathcal{T}}_t(f_t) = \overline{\mathcal{T}}_t(f_t) + w_t, \ \ \ t \in \{ 0, 1, 2, \dots \},
\end{equation}
where $w_t \in X$ is  {some additive} noise.
 {\begin{assumption}\label{assump:noise}
The additive noise is uniformly upper-bounded by a constant $\epsilon_w>0$, i.e., $\| w_t \| \leq \epsilon_w$ for all $t \in \{0, 1, 2, \dots \}$.
\end{assumption}}
% with the property that  with $\epsilon_w > 0$ being a bounded constant.
 {Note that the shape of the function $f_t$  {can change} over time and can be non-convex.
However, the following theorem shows that an upper bound can be established for the distance between $f_t$ and the time-{varying} fixed point $f_t^*$.
% The following theorem presents an upper bound on the distance between $f_t$ and the time-{varying} function $f_t^*$.
}

\begin{theorem}
\label{thm:contraction-expansion-noise}
Consider {arbitrary} time-varying probabilistic contraction-expansion mappings $\mathcal{T}_t$
%defined in Equation \eqref{probabilistic_contraction_expansion_t}
with fixed points $f_t^*$, where $\sup_t \big ( q_t^2 \cdot p_t + Q_t^2 \cdot (1 - p_t) \big ) < 1$  {for $t \in \{0, 1, 2, \dots \}$.}
% and $d \big ( f_t^*, f_{t + 1}^* \big ) < \epsilon_f$ 
{Let} the time-varying function $f_t$ evolve over time according to the time-varying noisy probabilistic transformation in \eqref{time_varying_map_2}.
 {Under Assumptions \ref{assump:consecutive_fixed_point} and \ref{assump:noise}, it holds that}
\begin{equation}
    \begin{aligned}
    d \big ( f_t, f_t^* \big )
    \leq P_t \cdot d \big ( f_0, f_0^* \big ) + S_t \cdot ( \epsilon_f + \epsilon_w ),
    \end{aligned}
\end{equation}
where $P_t = \left ( \prod_{i = 0}^{t - 1} B_i \right )$ and $S_t = \left ( 1 + \sum_{i = 1}^{t - 1} \prod_{j = 1}^{t - i} B_j \right )$ are random variables with independent random variables $B_t$ having the distribution
\begin{equation}\label{eq:B_t}
    B_t =
    \begin{cases}
      q_t & \text{w.p.} \quad p_t\\
      Q_t & \text{otherwise}
    \end{cases}.
\end{equation}
The means and variances of $P_t$ and $S_t$ are  {upper-bounded} as
\begin{equation}
    \begin{aligned}
    \mathbb{E} \left [ P_t \right ]
    & \leq \left ( \sup_t \big ( q_t \cdot p_t + Q_t \cdot (1 - p_t) \big ) \right )^t \xrightarrow[]{t \rightarrow \infty} 0, \\
    \Var(P_t)
    & \leq \left ( \sup_t \big ( q_t^2 \cdot p_t + Q_t^2 \cdot (1 - p_t) \big ) \right )^t \xrightarrow[]{t \rightarrow \infty} 0,
    \end{aligned}
\end{equation}
and
\begin{equation}
    \begin{aligned}
    \mathbb{E} \left [ S_t \right ]
    & \leq \frac{1}{1 - \sup_t \big ( q_t \cdot p_t + Q_t \cdot (1 - p_t) \big )}, \\
    \Var    (S_t)    &    \leq    \frac{\big ( \bar{q}^2 \cdot \bar{p} + \bar{Q}^2 \cdot (1 - \bar{p}) \big ) \cdot \big ( 1 + \bar{q} \cdot \bar{p} + \bar{Q} \cdot (1 - \bar{p}) \big )}{\big (    1    -    \bar{q}^2    \cdot    \bar{p}    -    \bar{Q}^2    \cdot    (1    -    \bar{p}) \big )    \cdot    \big ( 1    -    \bar{q}    \cdot    \bar{p}    -    \bar{Q}    \cdot    (1    -    \bar{p}) \big )},
    \end{aligned}
\end{equation}
where {$\bar{q}$, $\bar{Q}$, and $\bar{p}$ satisfy} $\bar{q} \cdot \bar{p} + \bar{Q} \cdot (1 - \bar{p}) \geq \sup_{t \geq 1} \mathbb{E}[B_t]$ and $\bar{q}^2 \cdot \bar{p} + \bar{Q}^2 \cdot (1 - \bar{p}) \geq \sup_{t \geq 1} \mathbb{E}[B_t^2]$.

\end{theorem}

\begin{proof}
Under the time-varying probabilistic contraction-expansion mappings with added noise functions introduced in \eqref{time_varying_map_2}, the distance between $f_t$ and $f_t^*$ can be  {upper-bounded} as
\begin{equation}
    \label{distance_at_time_t}
    \begin{aligned}
    d \big ( f_t, f_t^* \big ) 
    & = d \big ( \widetilde{\mathcal{T}}_{t - 1} \circ \cdots \circ \widetilde{\mathcal{T}}_0 (f_0), f_t^* \big )\\
    & \overset{(a)}{=} d \big ( \overline{\mathcal{T}}_{t - 1} \big ( \widetilde{\mathcal{T}}_{t - 2} \circ \cdots \circ \widetilde{\mathcal{T}}_0 (f_0) \big ) + w_{t - 1}, f_t^* \big ) \\
    & = \big \| \overline{\mathcal{T}}_{t - 1} \big ( \widetilde{\mathcal{T}}_{t - 2} \circ \cdots \circ \widetilde{\mathcal{T}}_0 (f_0) \big ) + w_{t - 1} - f_t^* \big \| \\
    & \overset{(b)}{\leq} d \big ( \overline{\mathcal{T}}_{t - 1} \big ( \widetilde{\mathcal{T}}_{t - 2} \circ \cdots \circ \widetilde{\mathcal{T}}_0 (f_0) \big ), f_t^* \big ) + \| w_{t - 1} \| \\
    & \overset{(c)}{\leq} d \big ( \overline{\mathcal{T}}_{t - 1} \big ( \widetilde{\mathcal{T}}_{t - 2}    \circ    \cdots    \circ    \widetilde{\mathcal{T}}_0 (f_0) \big ), f_{t - 1}^* \big )    +    d \big ( f_{t - 1}^*, f_t^* \big )    +    \| w_{t - 1} \| \\
    & \overset{(d)}{\leq} B_{t - 1} \cdot d \big ( \widetilde{\mathcal{T}}_{t - 2} \circ \cdots \circ \widetilde{\mathcal{T}}_0 (f_0), f_{t - 1}^* \big ) + \epsilon_f + \epsilon_w,
    \end{aligned}
\end{equation}
where  {$\circ$ denotes the composition of linear operators,} the definition of {the}  {mapping} $\widetilde{\mathcal{T}}_{t - 1}$ in \eqref{time_varying_map_2} is used in $(a)$, inequalities $(b)$ and $(c)$ are true by {the} triangular inequality, {and} $(d)$  {follows from Assumptions \ref{assump:consecutive_fixed_point} and \ref{assump:noise}} in addition to the probabilistic contraction-expansion {property of the} operator $\overline{\mathcal{T}}_{t - 1}$
%defined in Equation \eqref{probabilistic_contraction_expansion_t}
and the fact that $\overline{\mathcal{T}}_{t- 1}(f_{t - 1}^*) = f_{t - 1}^*$.
Furthermore, the independent random variables $B_t$ for $t \geq 0$ used in $(d)$ have the distribution  {as specified in \eqref{eq:B_t}.}
% \begin{equation}
% \label{B_t_distribution}
%     B_t =
%     \begin{cases}
%       q_t & \text{w.p.} \quad p_t\\
%       Q_t & \text{otherwise}
%     \end{cases}.
% \end{equation}
Taking similar steps as in \eqref{distance_at_time_t}, we have
\begin{equation}
    \label{distance_at_time_t_2}
    \begin{aligned}
    d \big ( f_t, f_t^* \big ) 
    &    \leq B_{t - 1}    \cdot    \left (    B_{t - 2}    \cdot    d \big ( \widetilde{\mathcal{T}}_{t - 3}    \circ    \cdots    \circ    \widetilde{\mathcal{T}}_0 (f_0), f_{t - 2}^* \big )    +    \epsilon_f    +    \epsilon_w    \right )    +    \epsilon_f    +    \epsilon_w \\
    & \leq B_{t - 1} \cdot \Big ( B_{t - 2} \cdot \Big ( B_{t - 3} \cdot d \big ( \widetilde{\mathcal{T}}_{t - 4} \circ \cdots \circ \widetilde{\mathcal{T}}_0 (f_0), f_{t - 3}^* \big ) + \epsilon_f + \epsilon_w \Big ) + \epsilon_f + \epsilon_w \Big ) + \epsilon_f + \epsilon_w \\
    & \leq \left ( \prod_{i = 0}^{t - 1} B_i \right ) \cdot d \big ( f_0, f_0^* \big ) + \left ( 1 + \sum_{i = 1}^{t - 1} \prod_{j = 1}^{t - i} B_j \right ) \cdot ( \epsilon_f + \epsilon_w ) \\
    & \leq P_t \cdot d \big ( f_0, f_0^* \big ) + S_t \cdot ( \epsilon_f + \epsilon_w ),
    \end{aligned}
\end{equation}
where $P_t = \left ( \prod_{i = 0}^{t - 1} B_i \right )$ and $S_t = \left ( 1 + \sum_{i = 1}^{t - 1} \prod_{j = 1}^{t - i} B_j \right )$ are random variables whose means and variances {will be} calculated below.
Using the independence of random variables $B_t$ for $t \geq 0$, we have
\begin{equation}
    \label{mean_P_t}
    \begin{aligned}
    \mathbb{E} \left [ P_t \right ]
    &    =    \mathbb{E}    \left [ \prod_{i = 0}^{t - 1} B_i \right ]   
    =    \prod_{i = 0}^{t - 1} \mathbb{E} \left [ B_i \right ]   
    =    \prod_{i = 0}^{t - 1} \big ( q_t \cdot p_t + Q_t \cdot (1 - p_t) \big )  \leq \left ( \sup_t \big ( q_t \cdot p_t + Q_t \cdot (1 - p_t) \big ) \right )^t
    \end{aligned}
\end{equation}
and
\begin{equation}
    \label{variance_P_t}
    \begin{aligned}
     \Var(P_t)
    &= \mathbb{E} \left [ P_t^2 \right ] - \left ( \mathbb{E} \left [ P_t \right ] \right )^2 \\
    & = \mathbb{E} \left [ \prod_{i = 0}^{t - 1} B_i^2 \right ] - \prod_{i = 0}^{t - 1} \big ( q_t \cdot p_t + Q_t \cdot (1 - p_t) \big )^2 \\
    & \leq    \prod_{i = 0}^{t - 1}    \big ( q_t^2    \cdot    p_t    +    Q_t^2    \cdot    (1    -    p_t) \big ) \\  
    & \leq    \left (    \sup_t    \big ( q_t^2    \cdot    p_t    +    Q_t^2    \cdot    (1    -    p_t) \big )    \right )^t    .
    \end{aligned}
\end{equation}
Note that it is already shown in  \eqref{eq:implication_of_assumption} that $q_t^2 \cdot p_t + Q_t^2 \cdot (1 - p_t) < 1$ {implies} $q_t \cdot p_t + Q_t \cdot (1 - p_t) < 1$,  {and therefore it suffices to assume} that $\sup_t \big ( q_t^2 \cdot p_t + Q_t^2 \cdot (1 - p_t) \big ) < 1$.
Furthermore,
\begin{equation}
    \label{mean_S_t}
    \begin{aligned}
    \mathbb{E} \left [ S_t \right ] &= \mathbb{E} \left [ 1 + \sum_{i = 1}^{t - 1} \prod_{j = 1}^{t - i} B_j \right ] \\
    & = 1 + \sum_{i = 1}^{t - 1} \prod_{j = 1}^{t - i} \mathbb{E} \left [ B_j \right ]\\
    &= 1 + \sum_{i = 1}^{t - 1} \prod_{j = 1}^{t - i} \big ( q_j \cdot p_j + Q_j \cdot (1 - p_j) \big ) \\
    & \leq 1 + \sum_{i = 1}^{t - 1} \left ( \sup_j \big ( q_j \cdot p_j + Q_j \cdot (1 - p_j) \big ) \right )^{t - i} \\
    & \leq \frac{1}{1 - \sup_j \big ( q_j \cdot p_j + Q_j \cdot (1 - p_j) \big )}
    \end{aligned}
\end{equation}
and
\begin{equation}
    \label{variance_S_t}
    \begin{aligned}
    & \hspace{4.5mm} \Var(S_t)
    = \Var \left (1 + \sum_{i = 1}^{t - 1} \prod_{j = 1}^{t - i} B_j \right )
    = \Var \left ( \sum_{i = 1}^{t - 1} \prod_{j = 1}^{t - i} B_j \right ) \leq \mathbb{E} \left [ \left ( \sum_{i = 1}^{t - 1} \prod_{j = 1}^{t - i} B_j \right )^2 \right ].
    \end{aligned}
\end{equation}
Consider the sequence of independent and identically distributed random variables $\bar{B}_t$ for $t \in \{1, 2, \dots\}$ that have the distribution
\begin{equation}
    \bar{B}_t =
    \begin{cases}
      \bar{q} & \text{w.p. } \bar{p}\\
      \bar{Q} & \text{otherwise}
    \end{cases}
\end{equation}
 {such that} $\mathbb{E} [\bar{B}_t] \geq \sup_{i \geq 1} \mathbb{E}[B_i]$ and $\mathbb{E} [\bar{B}_t^2] \geq \sup_{i \geq 1} \mathbb{E}[B_i^2]$.
Proceeding with \eqref{variance_S_t}, {one can write}
\begin{equation}
\label{eq:var_S}
    \begin{aligned}
    & \hspace{4.5mm} \Var(S_t)
    \leq \mathbb{E} \left [ \left ( \sum_{i = 1}^{t - 1} \prod_{j = 1}^{t - i} B_j \right )^2 \right ]
    \leq \mathbb{E} \left [ \left ( \sum_{i = 1}^{t - 1} \prod_{j = 1}^{t - i} \bar{B}_j \right )^2 \right ] \leq \mathbb{E} \left [ \left ( \sum_{i = 1}^{\infty} \prod_{j = 1}^{i} \bar{B}_j \right )^2 \right ] = \mathbb{E} \left [ \bar{S}^2 \right ],
    \end{aligned}
\end{equation}
where $\bar{S} = \sum_{i = 1}^{\infty} \prod_{j = 1}^{i} \bar{B}_j$. We have $\mathbb{E} [\bar{S}] = \frac{\bar{q} \cdot \bar{p} + \bar{Q} \cdot (1 - \bar{p})}{1 - \bar{q} \cdot \bar{p} - \bar{Q} \cdot (1 - \bar{p})}$ and $\bar{S} = \bar{B}_1 \cdot (1 + \bar{B}_2 + \bar{B}_2 \cdot \bar{B}_3 + \cdots) = \bar{B}_1 \cdot (1 + \tilde{S})$, where $\tilde{S}$ is independent of $B_1$, and {the} random variables $\bar{S}$ and $\tilde{S}$ are identically distributed but not independent of each other.
Taking expectation on both sides of $\bar{S}^2 = \bar{B}_1^2 \cdot (1 + \tilde{S})^2$, and using the independence of $\tilde{S}$ and $B_1$ and the fact that $\mathbb{E}[\bar{S}^2] = \mathbb{E}[\tilde{S}^2]$, we have
\begin{equation}
\label{eq:second_moment_S}
    \begin{aligned}
        & \mathbb{E}[\bar{S}^2] = \mathbb{E}[\bar{B}_1^2] \cdot \mathbb{E}[1 + 2\tilde{S} + \tilde{S}^2]  = \big (\bar{q}^2 \cdot \bar{p} + \bar{Q}^2 \cdot (1 - \bar{p}) \big ) \times \left (1 + \frac{2 \big ( \bar{q} \cdot \bar{p} + \bar{Q} \cdot (1 - \bar{p}) \big )}{1 - \bar{q} \cdot \bar{p} - \bar{Q} \cdot (1 - \bar{p})} + \mathbb{E}[\tilde{S}^2] \right )  \\
        \Longrightarrow\  & \mathbb{E}[\bar{S}^2]    =    \frac{\big ( \bar{q}^2 \cdot \bar{p} + \bar{Q}^2 \cdot (1 - \bar{p}) \big ) \cdot \big ( 1 + \bar{q} \cdot \bar{p} + \bar{Q} \cdot (1 - \bar{p}) \big )}{\big ( 1    -    \bar{q}^2    \cdot    \bar{p}    -    \bar{Q}^2    \cdot    (1    -    \bar{p}) \big )    \cdot    \big ( 1    -    \bar{q}    \cdot    \bar{p}    -    \bar{Q} \cdot (1    -    \bar{p}) \big )}  .
    \end{aligned}
\end{equation}
Putting \eqref{eq:var_S} and \eqref{eq:second_moment_S} together, {it can be concluded that} $\Var(S_t) \leq \frac{\big ( \bar{q}^2 \cdot \bar{p} + \bar{Q}^2 \cdot (1 - \bar{p}) \big ) \cdot \big ( 1 + \bar{q} \cdot \bar{p} + \bar{Q} \cdot (1 - \bar{p}) \big )}{\big ( 1 - \bar{q}^2 \cdot \bar{p} - \bar{Q}^2 \cdot (1 - \bar{p}) \big ) \cdot \big ( 1 - \bar{q} \cdot \bar{p} - \bar{Q} \cdot (1 - \bar{p}) \big )}$, {which completes} the proof. %\qquad \qquad \qquad \qquad \qquad \qquad \qquad \quad \quad \Halmos
%\begin{flushright}
%\Halmos
%\end{flushright}
\end{proof}

 {In the absence of the adversary, the probabilistic contraction-expansion mapping $\overline{\mathcal{T}}_t$ is purely a contraction with the rate $q_t$.
We obtain the following corollary as a direct consequence of Theorem \ref{thm:contraction-expansion-noise}.
\begin{corollary}
\label{cor:thm:time-varying_operator}
Consider {arbitrary} time-varying contraction mappings $\overline{\mathcal{T}}_t$ with {the} contraction constants $q_t$ and fixed points $f_t^*$.
{Suppose that $q = \sup_t q_t < 1$ and that Assumption \ref{assump:consecutive_fixed_point} holds.}
{{Let} the time-varying function $f_t$ evolve over time according to 
% the time-varying noisy mappings in 
 \eqref{time_varying_map_2}.}
For $\epsilon>0$, we define the hitting time as $T(\epsilon) = \min \left \{ T : d \big ( f_t, f_t^* \big ) < \epsilon, {\ \forall t \geq T} \right \}$.
{If $\epsilon \in (\frac{1}{1 - q} \cdot (\epsilon_f + \epsilon_w), \frac{1}{1 - q} \cdot (\epsilon_f + \epsilon_w) + D]$, then}
%introduced in Definition \ref{def:HT_TV}, with $\frac{1}{1 - q} \cdot (\epsilon_f + \epsilon_w) \leq \epsilon \leq \frac{1}{1 - q} \cdot (\epsilon_f + \epsilon_w) + D$ is given by
\begin{equation}
    T(\epsilon) \leq 1 + {\ln \left ( {\left ( \epsilon - \frac{1}{1 - q} \cdot (\epsilon_f + \epsilon_w) \right )}\Big/{D} \right )}\bigg/{\ln(q)},
\end{equation}
%\begin{equation}
 %   T(\epsilon) \leq 1 + \frac{\ln \Big ( \big ( \epsilon - \frac{1}{1 - q} \cdot (\epsilon_f + \epsilon_w) \big ) \big / D \Big )}{ \Big / \ln(q)},
%\end{equation}
where $\epsilon_w$ is an upper bound on the norm of each noise function and $D>0$ is an upper bound  {on $d \big ( f_0^*, f_0 \big )$.}
\end{corollary}
\begin{proof}
When the time-varying mappings $\{\mathcal{T}_t\}$ are only contraction mappings, the random variable $B_t$ is equal to $q_t$ with probability $1$ in \eqref{eq:B_t}.
As a result, Equation \eqref{distance_at_time_t_2} has the following form:
\begin{equation}\label{distance_at_time_t_1}
d \big ( f_t, f_t^* \big )\leq q^t \cdot d \big ( f_0, f_0^* \big ) + \frac{1}{1 - q} \cdot (\epsilon_f + \epsilon_w),
\end{equation}
where we use $q = \sup_t q_t$.
{Since} the right-hand side of \eqref{distance_at_time_t_1} is 
 {decreasing in $t$, the hitting time $T(\epsilon)$ is upper-bounded by the minimum value of $t$ that satisfies $q^t \cdot d \big ( f_0, f_0^* \big ) + \frac{1}{1 - q} \cdot (\epsilon_f + \epsilon_w) \leq \epsilon$.}
% the  is an upper bound {on} .
%in Definition~\ref{def:HT_TV}.
The proof is completed by noticing that $d \big ( f_0^*, f_0 \big )$ is upper-bounded by a constant $D > 0$ and $\frac{1}{1 - q} \cdot (\epsilon_f + \epsilon_w) \leq \epsilon \leq \frac{1}{1 - q} \cdot (\epsilon_f + \epsilon_w) + D$.
\end{proof}
{Corollary \ref{cor:thm:time-varying_operator} formalizes how many iterations are required in the value iteration with additive noise and a time-varying contraction operator -- that can be caused by a time-varying environment -- to guarantee that the ultimate function value is in an $\epsilon$-neighborhood of the fixed point.}
}
 {
\begin{remark}
Tighter bounds on the hitting time for Theorems \ref{thm:probabilistic_banach} and \ref{thm:contraction-expansion-noise} may be obtained by applying concentration inequalities involving higher moments instead of Chebyshev's inequality.
However, since our bounds already have logarithmic dependence on the relevant parameters $p$, $Q$, $L$, $\epsilon$, and $d(f_1,f_0)$, they are sufficient for most practical purposes as long as those parameters do not scale exponentially with the problem size.
% The bounds in .
\end{remark}
}

\subsection{Optimization of Time-Varying Functions with Additive Noise}
\label{section_optimization_TV_continuous_noise}
Consider the unknown time-varying continuous function $f_t: \mathcal{D} \rightarrow \mathcal{R}$ with {the known} bounded Lipschitz constant $K_t$,  {over the discrete-time horizon $t \in \{1, 2, \dots\}$}, where $\mathcal{D} \subset \mathbb{R}^d$ is a compact set and $\mathcal{R} \subset \mathbb{R}$.
The goal is to $\epsilon$-optimize the unknown time-varying function $f_t$, i.e., to find {a} possibly time-varying point $\widehat{x}_t^*$ such that  {$| f_t(\widehat{x}_t^*) - f_t(x_t^*) | \leq \epsilon$} for $\epsilon > 0$, where $x_t^* = \argmin_{x \in \mathcal{D}} f_t(x)$.
%and $\| \cdot \|$ is the $L^2$-norm throughout this article.
 {Although the function $f_t$ is unknown, inquiries of the function values at given input points can be made in consecutive rounds, which are evaluated with added noise.}
%  {with a zero mean} that are independent and identically distributed over time and different input points.
 {More precisely, at round $t \in \{1, 2, \dots\}$, we consider querying the function $f_t$ on the set of input points $\mathcal{P} = \{x_1, \dots, x_n \} \subset \mathcal{D}$,
%with $[n] = \{1, 2, \dots, n\}$,
and the revealed values are}
\begin{equation}
\label{eq_noise}
    \widetilde{f}_t(x_i) = f_t(x_i) + N_t(x_i),
\end{equation}
where $N_t(x_i)$  {is some noise satisfying the following assumption.
\begin{assumption}\label{assump:noise_N}
The noise parameters $N_t(x_i)$ are bounded i.i.d. random variables with zero mean, i.e., $\mathbb{E}[N_t(x_i)] = 0$, for which there exists $L_N >0$ such that $\left[\sup\{N_t(x_i)\}-\inf\{N_t(x_i)\}\right]<L_N$ for all $t\in\{1, 2, \dots\}$ and $x_i \in \mathcal{P}$.
\end{assumption}
}
% i.i.d. random variables.
% that are strictly bounded by an interval with length $L_N$ and $\mathbb{E}[N_t(x_i)] = 0$  for all $t \in \{1, 2, \dots\}$ and
%for all
% $x_i \in \mathcal{P}$.
If the noise is disruptive enough, a single set of observed noisy function values $f_t(x_i)$ for all $x_i \in \mathcal{P}$ may not represent the unknown target function accurately, making it impossible to $\epsilon$-optimize the function {with a few number of observations}.
Furthermore, {since} the function {changes} over time, old observations may not be useful in $\epsilon$-optimizing the time-varying function  {as $t$ increases}.
Putting these two facts into perspective, the estimate of the target function $f_t$ at round $t - 1$, {namely} $\widehat{f}_{t - 1}$, {may need to be} updated with the new observation at round $t$,
%$t \in \{2, 3, \dots\}$,
while discarding inaccurate old observations. {We propose the following formula for estimating $f_t$:}
\begin{equation}
\label{eq_update}
\begin{aligned}
    \widehat{f}_t(x_i) = & \frac{\min\{t, T + 1\} - 1}{\min\{t, T\}} \cdot \widehat{f}_{t - 1}(x_i) + \frac{1}{\min\{t, T\}} \cdot \widetilde{f}_t(x_i) - \frac{1}{T} \cdot \widetilde{f}_{t - T}(x_i) \cdot \mathbbm{1}\{t > T\},
\end{aligned}
\end{equation}
where $\mathbbm{1}\{\cdot\}$ is the indicator function.
The parameter $T$, whose value to be specified, {should be chosen such that old data is discarded} due to the time-varying nature of the function while not harming accurate estimation of the function value in the presence of noise.
The computational cost of \eqref{eq_update} is {on} the same order of that of {the} moving average update in reinforcement learning, but in  \eqref{eq_update} there is a need {for} storing the previous $T$ observations in order to have access to $\widetilde{f}_{t - T}(x_i)$.

The {estimation} function $\widehat{f}_t(x_i)$ changes over time and may not represent the target function {for small values of $t$}.
However, there {may exist a hitting time $T$ that is used in \eqref{eq_update} after which} optimizing the estimated function $\widehat{f}_t$ $\epsilon$-optimizes the target function $f_t$ with an associated confidence level $1 - a$, where $0 < a \leq 1$.
As a result, the complexity of $\epsilon$-optimizing the unknown time-varying target function $f_t$ {in long-run} is irrelevant to the complexity of optimizing function $\widehat{f}_t$ up to the hitting time $T$.
Consequently, the hitting time $T$ as well as the optimization complexity of $\widehat{f}_t$ for $t \geq T$ captures the difficulty of $\epsilon$-optimizing the target function $f_t$ rather than the cumulative optimization complexities of functions $\widehat{f}_t$ for $t < T$.
Formally speaking, the hitting time $T(\epsilon, a)$ is defined below.

\begin{definition}
\label{definition_hitting_time_continuous_noise}
{Given $\epsilon > 0$ and $a \in (0, 1]$,} the hitting time $T(\epsilon, a)$ is defined as
\begin{equation}
\label{hitting_time}
\begin{aligned}
    T(\epsilon, a)    =    \min    \Big \{    T  :   \mathbb{P} \big (   {\big | f_t(\widehat{x}_t^*)    -    f_t(x_t^*) \big |    \leq    \epsilon }\big )    \geq    1    -    a,    {\ \forall t    \geq    T    } \Big \}  ,
\end{aligned}
\end{equation}
where $\widehat{x}_t^* = \argmin_{x \in \mathcal{P}} \widehat{f}_t(x)$ and $x_t^* = \argmin_{x \in \mathcal{D}} f_t(x)$.
\end{definition}
%In the rest of this subsection, an upper bound on hitting time $T(\epsilon, a)$ is derived under an appropriate grid of the function domain.
%and show how the upper bound is related to the shape of the target function $f_t$.

%\section{The Hitting Time Analysis for a Noisy Time-Varying Function}
%\label{hitting_time_analysis}

%Consider a $\delta$-grid of the function domain, $\mathcal{P} = \{x_1, x_2, \dots, x_n\}$, where for any $x \in \mathcal{D}$ there exists $x_i \in \mathcal{P}$ such that $\|x_i - x\| \leq \delta$.
 {To make the time-varying problem amenable to optimization, we also make the following assumption about the set of input points $\mathcal{P}$.
\begin{assumption}\label{assump:granularity}
For a given $\epsilon >0$, the set of input points $\mathcal{P}=\{x_1, x_2, \dots, x_n\}$ is a $\delta$-uniform grid of the function domain $\mathcal{D}$ such that $\delta < \frac{2\epsilon}{7\sqrt{d}K}$, where $K = \sup_{t \geq 1} K_t$ with $K_t$ being the Lipschitz constant of function $f_t$.
\end{assumption}
Recall that being a $\delta$-uniform grid means that $\mathcal{P}$ satisfies two properties:
(i) ${\{x_i + \delta e_j, x_i - \delta e_j\}} \cap \mathcal{D} \in \mathcal{P}$ for all $i \in {\{1, \dots, n\}}$ and
%for all
$j \in {\{1, \dots, d\}}$, where $e_1, \dots, e_d$ are the standard basis of $\mathbb{R}^d$, and {(ii)} for {every} $x \in \mathcal{D}$ there exists $x_i \in \mathcal{P}$ such that $\|x_i - x\| \leq \sqrt{d} \delta/2$.
}
% Consider a $\delta$-uniform grid of the function domain, $\mathcal{P} = \{x_1, x_2, \dots, x_n\} {\subset \mathcal{D}}$, which means {that two properties hold: (i)} ${\{x_i + \delta e_j, x_i - \delta e_j\}} \cap \mathcal{D} \in \mathcal{P}$ for all $i \in {\{1, \dots, n\}}$ and
%for all
% $j \in {\{1, \dots, d\}}$, where $e_1, \dots, e_d$ are the standard basis of $\mathbb{R}^d$, and {(ii)} for {every} $x \in \mathcal{D}$ there exists $x_i \in \mathcal{P}$ such that $\|x_i - x\| \leq \sqrt{d} \delta/2$.
% Let the grid {have a fine granularity} in the sense that $\delta < \frac{2\epsilon}{7\sqrt{d}K}$, where $K = \sup_{t \geq 1} K_t$ with $K_t$ being the Lipschitz constant of function $f_t$.
 {
The fine granularity assumption, i.e., $\delta < \frac{2\epsilon}{7\sqrt{d}K}$, assures that there exists a grid point whose unknown function value at time $t$ is at least $\frac{\epsilon}{7}$ close to the minimum of function $f_t$.}
Denote such points of the grid $\mathcal{P}$ by $\mathcal{N}_t(\frac{\epsilon}{7}) = \{x_i \in \mathcal{P}: f_t(x_i) - f_t(x_t^*) \leq \frac{\epsilon}{7} \}$ and let $\overline{\mathcal{N}}_t(\epsilon) = \{x_i \in \mathcal{P}: f_t(x_i) - f_t(x_t^*) > \epsilon \}$.
 {Without loss of generality, we assume that $\overline{\mathcal{N}}_t(\epsilon) \neq \emptyset$;} otherwise, any point in $\mathcal{P}$ $\epsilon$-optimizes function $f_t$.
%%%Let $\Delta_t(x_i) = \| f_t(x_i) - f_t(x_t^*) \|$ for grid point $x_i \in \mathcal{P}$, which results in $\Delta_t(x_i) > \epsilon$ for all $x_i \in \overline{\mathcal{N}}_t(\epsilon)$ by construction; as a result, $\Delta_m = \min_{\substack{x_i \in \overline{\mathcal{N}}_t(\epsilon), \ t \geq 1}} \Delta_t(x_i) > \epsilon$ meaning $\Delta_m > 0$.
The following theorem presents an upper bound on {the} hitting time.
%defined in Definition \ref{definition_hitting_time_continuous_noise}.

\begin{theorem}
\label{theorem_hitting_time}
Consider the unknown time-varying function $f_t$ with the property  {
% \begin{equation}
    ${| f_t(x) - f_{t - 1}(x) |} \leq \frac{\epsilon^3}{43 L_N^2 \cdot \ln(\frac{n}{a})}$, for all $t \geq 1$ and $ x \in \mathcal{D}$.
% \end{equation}
Given $\epsilon > 0$ and $a \in (0, 1]$, let Assumptions \ref{assump:noise_N} and \ref{assump:granularity} hold. Then, the hitting time $T(\epsilon, a)$ satisfies the inequality}
%%%in Definition \ref{definition_hitting_time_continuous_noise}
%for the estimate function $\widehat{f}_t$ introduced in Equation \eqref{eq_update}
\begin{equation}\label{eq:thm_hitting_time}
    T(\epsilon, a) \leq \frac{49L_N^2}{8 \epsilon^2} \cdot \ln \left ( \frac{n}{a} \right ) + 1.
\end{equation}

\end{theorem}

\begin{proof}
In order to find an upper bound on the hitting time $T(\epsilon, a)$,
%in Definition \ref{definition_hitting_time_continuous_noise},
 {it is reasonable to assume that the function {variation over time} is {upper-bounded}; otherwise, there may not be enough time for learning the rapidly changing functions $\{f_t\}$.}
{Assume that the time-variation} of the unknown time-varying target function $f_t$ is  {upper-bounded} by
\begin{equation}
\label{time_varying_constraint}
     {| f_t(x) - f_{t - 1}(x) |} \leq \frac{\epsilon}{7T}, \quad \forall t \geq 1, \forall x \in \mathcal{D}.
\end{equation}
Then,  {under Assumption \ref{assump:granularity}}, the hitting event  {defined} in \eqref{hitting_time} satisfies the following condition
\begin{equation}
\label{eq_subset}
    \begin{aligned}
        & \Bigg \{ \exists x_i \in \mathcal{N}_t(\frac{\epsilon}{7}) \text{ such that } \frac{1}{T} \cdot \hspace{-0.3cm}\sum_{s = t - T + 1}^t N_s(x_i) \leq \frac{2\epsilon}{7} \textbf{ and }
        \frac{1}{T} \cdot \hspace{-0.3cm} \sum_{s = t - T + 1}^t N_s(x_i) \geq -\frac{2\epsilon}{7}, \forall x_i \in \overline{\mathcal{N}}_t(\epsilon) \Bigg \} \\
        \subseteq & \left \{  {\big| f_t(\widehat{x}_t^*) - f_t(x_t^*) \big |} \leq \epsilon \right \}, \quad \forall t \geq T.
    \end{aligned}
\end{equation}
The above equation {holds} true because \eqref{eq_noise} and \eqref{eq_update} result in $\widehat{f}_t(x_i) = \frac{1}{T} \cdot \sum_{s = t - T + 1}^t f_s(x_i) + \frac{1}{T} \cdot \sum_{s = t - T + 1}^t N_s(x_i)$ for $t \geq T$, and by \eqref{time_varying_constraint}, {one can write}
\begin{equation}
\label{value_difference}
\begin{aligned}
    & \widehat{f}_t(x_i) \leq f_t(x_i) + \frac{\epsilon}{7} + \frac{1}{T} \cdot \sum_{s = t - T + 1}^t N_s(x_i), \quad \forall x_i \in \mathcal{N}_t(\frac{\epsilon}{7}), \\
    & \widehat{f}_t(\overline{x}_j) \geq f_t(\overline{x}_j) - \frac{\epsilon}{7} + \frac{1}{T} \cdot \sum_{s = t - T + 1}^t N_s(\overline{x}_j), \quad \forall \overline{x}_j \in \overline{\mathcal{N}}_t(\epsilon).
\end{aligned}
\end{equation}
Furthermore, $f_t(\overline{x}_j) - f_t(x_i) > \frac{6\epsilon}{7}$ for all $\overline{x}_j \in \overline{\mathcal{N}}_t(\epsilon)$ and
%for all
$x_i \in \mathcal{N}_t(\frac{\epsilon}{7})$.
Taking the difference of the two inequalities in \eqref{value_difference} yields that $\widehat{f}_t(\overline{x}_j) - \widehat{f}_t(x_i) > \frac{4\epsilon}{7} + \sum_{s = t - T + 1}^t N_s(\overline{x}_j) - \sum_{s = t - T + 1}^t N_s(x_i)$.
If the event on the left-hand side of \eqref{eq_subset} is true, {then} $\widehat{f}_t(\overline{x}_j) - \widehat{f}_t(x_i) > 0$, which means {that} there exists $\widetilde{x}_t^* \in \mathcal{N}_t(\frac{\epsilon}{7})$ whose estimated function value is less than the estimated function value {at all points} $\overline{x}_j \in \overline{\mathcal{N}}_t(\epsilon)$.
Note that the estimated function value {at} a point $\overline{x}_t^* \in \mathcal{P} \setminus \left ( \mathcal{N}_t(\frac{\epsilon}{7}) \cup \overline{\mathcal{N}}_t(\epsilon) \right )$ can be less than $\widehat{f}_t(\widetilde{x}_t^*)$, but such a point also $\epsilon$-optimizes the function $f_t$.
Hence, $\widehat{x}_t^* = \argmin_{x \in \mathcal{P}} \widehat{f}_t(x)$ $\epsilon$-optimizes the function $f_t$, which means {that} the event on right-hand side of \eqref{eq_subset} is true.

 {Denote the event on the left-hand side of \eqref{eq_subset} as $E_t$, whose probability can be lower-bounded as
}
% The probability of the event  {can be} lower bounded as
\begin{equation}
\label{eq_left_hand_side}
     {
    \begin{aligned}
     \mathbb{P}\{E_t\}
        &{\overset{(a)}{\geq}}   %\prod_{x_i \in \mathcal{N}_t(\epsilon)}
        \mathbb{P} \left \{ \frac{1}{T} \cdot \sum_{s = t - T + 1}^t N_s(x_i) \leq \frac{2\epsilon}{7}, x_i \in \mathcal{N}_t(\frac{\epsilon}{7}) \right \}  \times \prod_{x_i \in \overline{\mathcal{N}}_t(\epsilon)} \mathbb{P} \left \{ \frac{1}{T} \cdot \sum_{s = t - T + 1}^t N_s(x_i) \geq -\frac{2\epsilon}{7} \right \} \\
        &{\overset{(b)}{\geq}}  \ \prod_{x_i \in \mathcal{P}} \left ( 1 -  \exp \left ( - \frac{8 T \epsilon^2}{49L_N^2} \right ) \right ) \\
        % > 1 - \sum_{i = 1}^n \exp \left ( - \frac{8 T \epsilon^2}{49L_N^2} \right ) \\
        &{>}  \ 1 - n \cdot \exp \left ( - \frac{8 T \epsilon^2}{49L_N^2} \right ),
    \end{aligned}
    }
\end{equation}
where $(a)$ is true as the added noise signals are independent of each other and $(b)$ {follows} from Hoeffding's inequality and possibly multiplying by positive terms that are less than one.
Putting \eqref{eq_subset} and \eqref{eq_left_hand_side} together, we have
\begin{equation}
\label{eq_lower_bound_prob}
    \begin{aligned}
        & \mathbb{P}    \left \{  {\big | f_t(\widehat{x}_t^*)    -    f_t(x_t^*) \big |}    \leq    \epsilon \right \}    \geq    1 - n \cdot \exp    \left (    - \frac{8 T \epsilon^2}{49L_N^2} \right )  , \forall t \geq T.
    \end{aligned}
\end{equation}
If $1 - n \cdot \exp \left ( - \frac{8 T \epsilon^2}{49L_N^2} \right ) \geq 1 - a$ or equivalently $T \geq \frac{49L_N^2}{8\epsilon^2} \cdot \ln \left ( \frac{n}{a} \right ) $, we have
\begin{equation}
\label{eq_lower_bound_prob_2}
    \begin{aligned}
        \mathbb{P} \left \{  {\big| f_t(\widehat{x}_t^*) - f_t(x_t^*) \big |} \leq \epsilon \right \} \geq 1 - a, \quad \forall t \geq T.
    \end{aligned}
\end{equation}
As a result, an upper bound on {the} hitting time $T(\epsilon, a)$ defined in \eqref{hitting_time} is provided as
\begin{equation}
    T(\epsilon, a) \leq \frac{49L_N^2}{8 \epsilon^2} \cdot \ln \left ( \frac{n}{a} \right ) + 1.
\end{equation}
 {We substitute the upper bound on $T(\epsilon,a)$ into \eqref{time_varying_constraint}. It follows that the above analysis is valid if}
\begin{equation}
     {| f_t(x) - f_{t - 1}(x) |} \leq \frac{8 \epsilon^3}{343 L_N^2 \cdot \ln(\frac{n}{a})}, \quad \forall t \geq 1, \forall x \in \mathcal{D}.
\end{equation}
 {This completes the proof.}
% {and thus} the analysis holds if $\| f_t(x) - f_{t - 1}(x) \| \leq \frac{ \epsilon^3}{43 L_N^2 \cdot \ln(\frac{n}{a})} < \frac{8 \epsilon^3}{343 L_N^2 \cdot \ln(\frac{n}{a})}$. %\qquad \qquad \qquad \qquad \Halmos
%{and thus} the analysis holds if {the variation of the function over time} is less than or equal to $\frac{ \epsilon^3}{43 L_N^2 \cdot \ln(\frac{n}{a})} < \frac{8 \epsilon^3}{343 L_N^2 \cdot \ln(\frac{n}{a})}$.
%\begin{flushright}
%\Halmos
%\end{flushright}
\end{proof}

\begin{remark}
Note that the cardinality of the $\delta$-grid with $\delta < \frac{2\epsilon}{7\sqrt{d}K}$ used in Theorem \ref{theorem_hitting_time}, {namely} $n = |\mathcal{P}|$, depends on $\epsilon$.
As an example, if $\mathcal{D}$ can be written as the Cartesian product of $d$ intervals {of} length at most $M$ as $\mathcal{D} = \mathcal{D}_1 \times \mathcal{D}_2 \times \dots \times \mathcal{D}_d$, {then} the cardinality of the $\delta$-grid would be $n = \mathcal{O} \left ( \left (\frac{\sqrt{d}KM}{\epsilon}\right )^d \right )$, {and therefore} the upper bound {on} the hitting time in Theorem \ref{theorem_hitting_time} is given by $T(\epsilon, a) \leq \mathcal{O} \left ( \frac{dL_N^2}{\epsilon^2} \cdot \ln \left ( \frac{\sqrt{d}KM}{\sqrt[d]{a} \epsilon} \right ) \right )$.
\end{remark}

{Theorem \ref{theorem_hitting_time} determines how fast the unknown function $f_t$ is allowed to change over time such that one can still learn the estimation function $\widehat{f}_t$ which is used to $\epsilon$-optimize the target function $f_t$ with a confidence level.
The parameter $T$ in \eqref{eq_update} can be set to the upper bound provided in Theorem \ref{theorem_hitting_time} so that old inaccurate observations are discarded and at the same time enough observations are used for an accurate estimation of $f_t$.
}

\subsection{Improved Bounds for Convex Functions}
\label{sec:improved_bounds_cts}
Consider the same framework as in Section \ref{section_optimization_TV_continuous_noise} {under additional assumptions to be stated here}.
Let $f_t$ be a convex function for all $t \geq 1$.
Denote the lower contour set of the convex function $f_t$ by $C_t(c) = \{x \in \mathcal{D}: f_t(x) - f_t(x_t^*) \leq c\}$ and the level set of the convex function $f_t$ by $L_t(c) = \{x \in \mathcal{D}: f_t(x) - f_t(x_t^*) = c\}$ for $c > 0$.
Define $\overline{C}_t(c_1, c_2) = \{x \in \mathcal{D}: c_1 < f_t(x) - f_t(x_t^*) \leq c_2\}$ when $c_2 > c_1$.
%Overloading notation, we use somehow similar notations to the previous subsection,
{Let}
$\mathcal{M}_t(c) = \{ x_i \in \mathcal{P}: x_i \in C_t(c) \}$ and $\overline{\mathcal{M}}_t(c_1, c_2) = \{ x_i \in \mathcal{P}: x_i \in \overline{C}_t(c_1, c_2) \}$.
 {
\begin{assumption}\label{assump:homeomorphic}
There exists $M>0$ such that $L_t(M)$ is homeomorphic to a $d$-dimensional sphere and is inside $\mathcal{D}$ for all $t \geq 1$.
\end{assumption}
}
If $d = 1$ or {$d = 2$}, a {sphere} is defined as two distinctive points {or a circle, respectively}.
{Note that a lower bound on $M$ can be estimated up to a precision with high probability, but $M$ is assumed to be known to simplify the proof concepts.}
%Although a lower bound on the maximum value whose level set is a closed plane and is inside the function domain can be estimated up to a precision with high probability, it is considered such a value is known to be $M$ to simplify the proof concepts, so $L_t(M)$ is a closed plane and is inside $\mathcal{D}$ for all $t \geq 1$.
 {
\begin{assumption}\label{assump:gradient_lower_bound}
There exists $k>0$ such that 
% \begin{equation}
    $\left \| \nabla f_t(x) \right \| \geq k$, for all $t\geq 1$ and $x \in \mathcal{D} \setminus C_t(\epsilon)$.
% \end{equation}
\end{assumption}
Intuitively, Assumption \ref{assump:gradient_lower_bound} requires every convex function $f_t$ have enough curvature inside its lower contour set $C_t(\epsilon)$, so that  $\| \nabla f_t(x) \|$ can be uniformly lower-bounded by a positive constant $k$ in $\mathcal{D} \setminus C_t(\epsilon)$ for all $t\geq 1$.}
% {Let $k > 0$ be a lower bound on} the norm of {the} gradient of convex function $f_t$ over the set $\mathcal{D} \setminus C_t(\epsilon)$ for all $t \geq 1$, {and assume $k$ is known}, i.e.,

Leveraging the new assumptions on {the} time-varying functions  {$\{f_t\}$}, the following theorem presents a tighter upper bound on {the} hitting time compared to Theorem \ref{theorem_hitting_time}.

\begin{theorem}
\label{theorem_hitting_time_convex}
Consider the unknown time-varying convex function $f_t$ with {the} property
$ {| f_t(x) - f_{t - 1}(x) |} \leq \frac{ \epsilon^3}{43 L_N^2 \cdot \ln(\frac{n}{a})}$,
%$\| f_t(x) - f_{t - 1}(x) \| \leq \frac{\epsilon}{7 T(\epsilon, a)}$,
for all $t \geq 1$ and
%for all
$x \in \mathcal{D}$.
 {Given $\epsilon > 0$ and $a \in (0, 1]$, suppose that Assumptions \ref{assump:noise_N}-\ref{assump:gradient_lower_bound} hold.
% {Assume that} there exists $M > \epsilon$ such that the level set $L_t(M)$ is {homeomorphic to a $d$-dimensional sphere} and is inside $\mathcal{D}$ for all $t \geq 1$, and {that} $\left \| \nabla f_t(x) \right \| \geq k$, for all $ t \geq 1$ {and} for all $x \in \mathcal{D} \setminus C_t(\epsilon)$.
Then, the hitting time $T(\epsilon, a)$ is upper-bounded by the minimum $T$ satisfying the inequality}
% {An} upper bound {on} the 
%%%%in Definition \ref{definition_hitting_time_continuous_noise}
%for the estimate function $\widehat{f}_t$ introduced in Equation \eqref{eq_update}
% is the  satisfying {the inequality}
\begin{equation}
\label{eq:thm:convex}
    \sum_{l = 0}^{l_m} n_l \cdot \exp \Big ( - \frac{2 T \big ( l + \frac{2}{7} \big )^2 \epsilon^2}{L_N^2} \Big ) \leq  a,
\end{equation}
where { $\sum_{l = 0}^{l_m} n_l = n$ and $l_m \leq \lfloor \frac{M}{\epsilon} \rfloor - 3$ such that $n_l = \frac{m_l}{1 + m_l} \cdot n + 1$  for $l \in \{0, 1, \dots, l_m - 1 \}$ with $m_l = \frac{2^{d+ 1} \cdot K \cdot \epsilon}{k \cdot \big (M - (l + 4) \epsilon \big )}$}.
%%%%{Note that} the last non-zero $n_l$ is not a free parameter {since the sum of all $n_l$ should be $n$.}
%and is determined by the previous upper bound values rather than the provided formula.
\end{theorem}

\begin{proof}
 {Following the same logic as in \eqref{eq_subset} and leveraging the convexity of $\{f_t\}$, we obtain that the the hitting event in \eqref{hitting_time} satisfies the  condition}
% Equation \eqref{eq_subset} can be improved as
\begin{equation}
\label{eq_subset_convex}
    \begin{aligned}
        & \bigg \{ \exists x_i \in \mathcal{M}_t(\frac{\epsilon}{7}) \text{ such that } \frac{1}{T} \cdot \hspace{-0.3cm}\sum_{s = t - T + 1}^t \hspace{-0.3cm} N_s(x_i) \leq \frac{2\epsilon}{7} \textbf{ and }  \frac{1}{T} \cdot \hspace{-0.3cm}\sum_{s = t - T + 1}^t \hspace{-0.3cm} N_s(x_i) \geq -\frac{2\epsilon}{7}, \forall x_i \in \overline{\mathcal{M}}_t \Big (\epsilon, 2 \epsilon \Big ) \textbf{ and } \\ 
        & \ \  
        \frac{1}{T}    \cdot\hspace{-0.3cm} \sum_{s = t - T + 1}^t N_s(x_i) \geq - \left ( l + \frac{2}{7} \right ) \epsilon, \forall x_i \in \overline{\mathcal{M}}_t \Big ( (l + 1) \epsilon, (l + 2)\epsilon\Big ), \forall 1 \leq l \leq \Big \lfloor \frac{M}{\epsilon} \Big \rfloor \bigg \} \\
        \subseteq & \left \{  {\big| f_t(\widehat{x}_t^*) - f_t(x_t^*) \big |} \leq \epsilon \right \}, \quad \forall t \geq T.
    \end{aligned}
\end{equation}

 {Denote the event on the left-hand side of \eqref{eq_subset_convex} as $E_t$, whose probability can be lower-bounded as}
\begin{equation}
\label{eq_left_hand_side_convex}
 {{\small\begin{aligned}
         \mathbb{P}\{E_t\}
        &{\overset{(a)}{\geq}}   %\prod_{x_i \in \mathcal{N}_t(\epsilon)}
        \mathbb{P} \left \{ \frac{1}{T} \cdot \sum_{s = t - T + 1}^t N_s(x_i) \leq \frac{2\epsilon}{7}, x_i \in \mathcal{M}_t(\frac{\epsilon}{7}) \right \} \times \prod_{x_i \in \overline{\mathcal{M}}_t \big (\epsilon, 2 \epsilon \big )} \mathbb{P} \left \{ \frac{1}{T} \cdot \sum_{s = t - T + 1}^t N_s(x_i) \geq -\frac{2\epsilon}{7} \right \}  \\
        &\hspace{0.5cm} \times\prod_{l = 1}^{ \lfloor \frac{M}{\epsilon} \rfloor} \prod_{x_i \in \overline{\mathcal{M}}_t \big ((l + 1) \epsilon, (l + 2) \epsilon \big )}   \mathbb{P}    \left \{    \frac{1}{T}    \cdot \hspace{-2.5mm} \sum_{s = t - T + 1}^t   N_s(x_i) \geq - \big (l + \frac{2}{7} \big ) \epsilon    \right \} \\
    &{\overset{(b)}{\geq}}     \left [    1    -    \exp    \left (    - \frac{8 T \epsilon^2}{49L_N^2} \right )    \right ]^{\overline{n}_0 + 1}   \times \prod_{l = 1}^{l_m}    \left [    1    -    \exp    \left (    - \frac{2 T \big ( l    +    \frac{2}{7} \big )^2    \epsilon^2}{L_N^2}    \right )    \right ]^{n_l} \\
        &{\geq}  \ 1 - \sum_{l = 0}^{l_m} n_l \cdot \exp \left ( - \frac{2 T \big ( l + \frac{2}{7} \big )^2 \epsilon^2}{L_N^2} \right )
    \end{aligned}
    }}
\end{equation}
where $(a)$ is true as the added noise signals are independent of each other and $(b)$ {follows from} Hoeffding's inequality, $\overline{n}_0$ is an upper bound on the number of grid points in the set $\overline{\mathcal{M}}_t \big (\epsilon, 2 \epsilon \big )$ and $n_0 = \overline{n}_0 + 1$, and $n_l$ is an upper bound on the number of grid points in the set $\overline{\mathcal{M}}_t \big ((l + 1)\epsilon, (l + 2)\epsilon \big )$,
%given $\overline{n}_0$ and $n_{\overline{l}}$ for $\overline{l} < l$, for all $0 \leq l \leq l_m$
where $l_m$ satisfies $\sum_{l = 0}^{l_m} n_l = n$ and $l_m \leq \lfloor \frac{M}{\epsilon} \rfloor - 3$.
Note that the last nonzero $n_l$ is not a free parameter {since the sum of all $n_l$ should be $n$.}
%the last non-zero $n_l$ is not a free parameter and is determined by the previous upper bound values, which are computed in the last part of the proof.
Putting \eqref{eq_subset_convex} and \eqref{eq_left_hand_side_convex} together, we have $\mathbb{P} \left \{  {\big| f_t(\widehat{x}_t^*) - f_t(x_t^*) \big |} \leq \epsilon \right \} \geq 1 - a$ for all $ t \geq T$ {provided that}
\begin{equation}
\label{eq_lower_bound_prob_convex}
    \begin{aligned}
        \sum_{l = 0}^{l_m} n_l \cdot \exp \Big ( - \frac{2 T \big ( l + \frac{2}{7} \big )^2 \epsilon^2}{L_N^2} \Big ) \leq  a,
    \end{aligned}
\end{equation}
which provides an upper bound on {the} hitting time $T(\epsilon, a)$ defined in \eqref{hitting_time}.
As stated earlier {in \eqref{time_varying_constraint}}, the above analysis is true if $ {| f_t(x) - f_{t - 1}(x) |} \leq \frac{\epsilon}{7 T(\epsilon, a)}$ for all $t \geq 1$ and $x \in \mathcal{D}$.
Using the general upper bound on {the} hitting time provided in Theorem~\ref{theorem_hitting_time}, the analysis holds if $ {| f_t(x) - f_{t - 1}(x) |} \leq \frac{ \epsilon^3}{43 L_N^2 \cdot \ln(\frac{n}{a})}$ for all $t \geq 1$ and $x \in \mathcal{D}$.

In the rest of the proof, the values of $n_l$ for $0 \leq l \leq l_m$ are computed.
The key idea{s} behind finding these upper bounds {are} that the level sets $\overline{L}_t \big ( (l + 1 ) \epsilon \big )$ for $0 \leq l \leq l_m + 2$ are nested {surfaces that are homeomorphic to a $d$-dimensional sphere} inside the function domain and {that} the minimum distance between any point of {a} level set from any of the other level set is controlled by $K$ and $k$.
Let $Vol(\cdot)$ {denote} the volume of an input $d$-dimensional set and $A(\cdot)$ {denote} the area of an input $(d - 1)$-dimensional {surface}. {By convention,} the area of a { {$d$-dimensional} sphere for $d = 1$ and $d = 2$} is equal to {2 and the length of the sphere, respectively}.
%\begin{equation}
%    \begin{aligned}
%        \overline{n}_0 & \leq \frac{2^d \cdot Vol\left ( C_t\big ( \epsilon, 3 \epsilon \big ) \right )}{\delta^d} \leq
%        \frac{2^d \cdot \frac{2 \epsilon}{k} \cdot A\left ( P_t \big ( \epsilon, 3 \epsilon \big ) \right )}{\delta^d}, \\
%        \sum_{l = 1}^{l_m} n_l & \geq \frac{Vol\left ( C_t\big (3 \epsilon, M - \epsilon \big ) \right )}{\delta^d} \geq
%        \frac{\frac{M - 4 \epsilon}{K} \cdot A\left ( P_t \big ( 3 \epsilon, M - \epsilon \big ) \right )}{\delta^d},
%    \end{aligned}
%\end{equation}
%where $P_t\big ( \epsilon, 3 \epsilon \big ) \subset C_t\big ( \epsilon, 3 \epsilon \big )$
%and $P_t \big ( 3 \epsilon, M - \epsilon \big ) \subset C_t \big ( 3 \epsilon, M - \epsilon \big )$ are two $(d - 1)$-dimensional planes
%such that $A \left ( P_t\big ( \epsilon, 3 \epsilon \big ) \right ) \leq A\left ( L_t\big ( 3 \epsilon \big ) \right ) \leq A \left ( P_t \big ( 3 \epsilon, M - \epsilon \big ) \right )$.
%Then,
%\begin{equation}
%    \frac{n_0 - 1}{n - n_0} = \frac{\overline{n}_0}{\sum_{l = 1}^{l_m} n_l} \leq
%    \frac{2^{d + 1} \cdot K \cdot \epsilon}{k \cdot \big (M - 4 \epsilon \big )} = p_0 
%    \ \Longrightarrow \ n_0 \leq \frac{p_0}{1 + p_0} \cdot n + 1.
%\end{equation}
%Similarly, the upper bound $n_l$ for $1 \leq l \leq l_m$
%given the values of $\overline{n}_0$ and $n_{\overline{l}}$ for $0 \leq \overline{l} < l$ can be derived as
{For every $l \in \{0, 1, \dots, l_m\}$, one can write}
%for $0 \leq l \leq l_m$
\begin{equation}
    \begin{aligned}
        & n_l - 1 \leq \frac{2^d \cdot Vol\left ( C_t\big ( (l + 1)\epsilon, (l + 3) \epsilon \big ) \right )}{\delta^d} 
        \leq  \frac{2^d \cdot \frac{2\epsilon}{k} \cdot A\left ( P_t \big ( (l + 1)\epsilon, (l + 3) \epsilon \big ) \right )}{\delta^d}, \\
        & \sum_{\overline{l} = l + 1}^{l_m} n_{\overline{l}} \geq \frac{Vol \left ( C_t \big ((l + 3) \epsilon, M - \epsilon \big ) \right )}{\delta^d} 
        \geq  \frac{\frac{M - (l + 4) \epsilon}{K} \cdot A\left ( P_t \big ( (l + 3) \epsilon, M - \epsilon \big ) \right )}{\delta^d},
    \end{aligned}
\end{equation}
where {the term} $2^d$ comes from the  {facts} that each $d$-dimensional cube  {has} at most $2^d$ endpoints and $P_t \big ( (l + 1)\epsilon, (l + 3) \epsilon \big ) \subset C_t\big ( (l + 1)\epsilon, (l + 3) \epsilon \big )$
and $P_t \big ((l + 3) \epsilon, M - \epsilon \big ) \subset C_t \big ((l + 3) \epsilon, M - \epsilon \big )$ are two $(d - 1)$-dimensional planes
such that $A \left ( P_t \big ( (l + 1) \epsilon, (l + 3) \epsilon \big ) \right ) \leq A\left ( L_t \big ( (l + 3) \epsilon \big ) \right ) \leq A \left ( P_t \big ((l + 3) \epsilon, M - \epsilon \big ) \right )$.
Then,
\begin{equation}
\begin{aligned}
    & \frac{n_l - 1}{n - n_l} \leq \frac{n_l - 1}{\sum_{\overline{l} = l + 1}^{l_m} n_{\overline{l}}} \leq
    \frac{2^{d+ 1} \cdot K \cdot \epsilon}{k \cdot \big (M - (l + 4) \epsilon \big )} = m_l 
    \ \Longrightarrow n_l \leq \frac{m_l}{1 + m_l} \cdot n + 1,
\end{aligned}
\end{equation}
%\begin{flushright}
%\Halmos
%\end{flushright}
 {which completes the proof.}
\end{proof}

\begin{remark}
 {We note that, since the left-hand side of \eqref{eq:thm:convex} is monotone decreasing in $T$, a number $T$ satisfying \eqref{eq:thm:convex} always exists.
By substituting the bound in \eqref{eq:thm_hitting_time} into \eqref{eq:thm:convex}, it can be verified that Theorem \ref{theorem_hitting_time_convex} provides a better bound than Theorem \ref{theorem_hitting_time} since some properties of convex functions are leveraged.}
A comparison of the results {of} Theorems \ref{theorem_hitting_time} and \ref{theorem_hitting_time_convex} {along with the simulation details is} depicted in Figure \ref{figure_TV_epsilon}.

\begin{figure}
\centering
\begin{subfigure}{.48\textwidth}
  \centering
  \includegraphics[width=\linewidth]{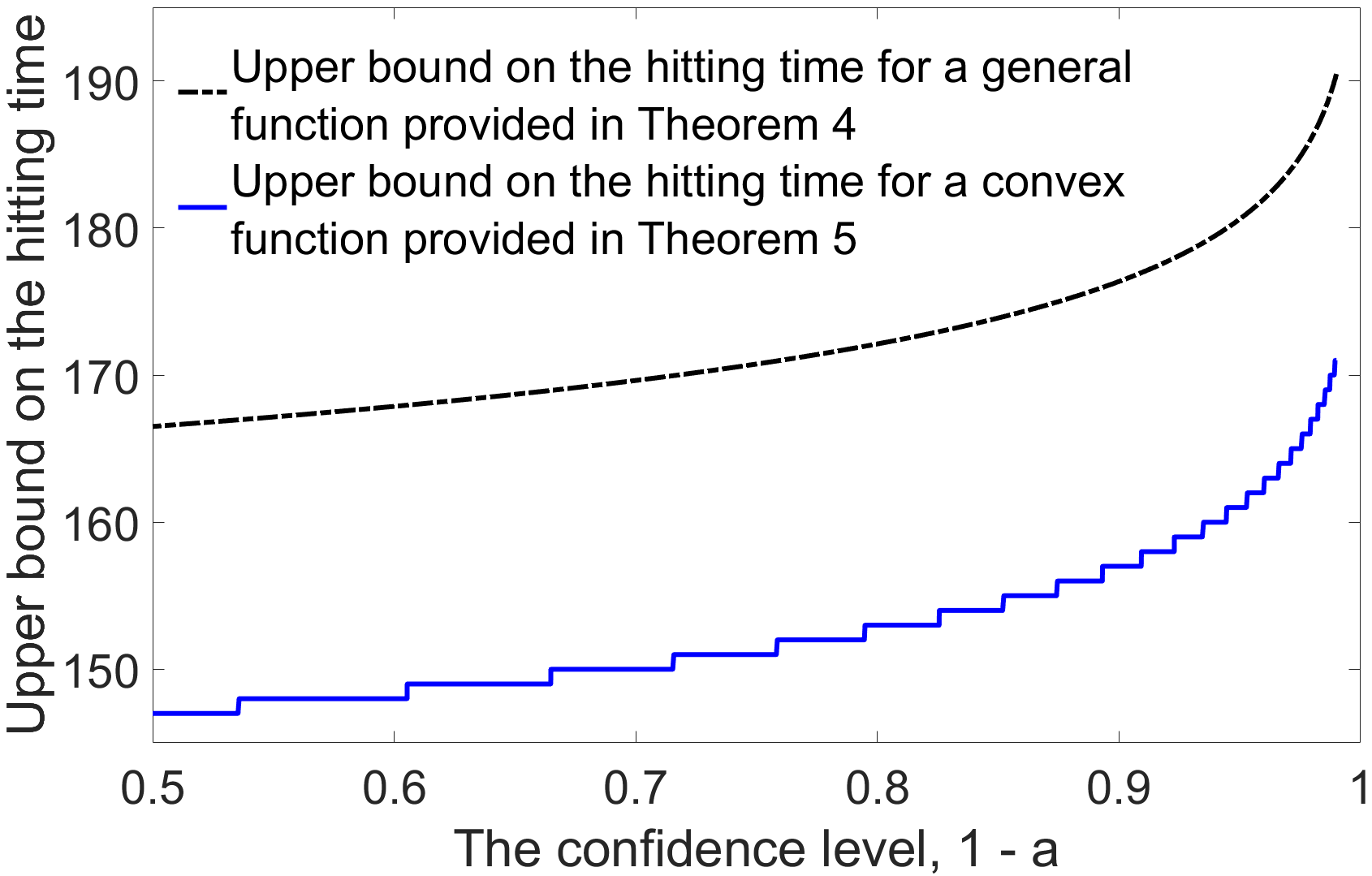}
  \caption{$\epsilon = L = 10^{-4}, n = 3.2 \times 10^{11}$, and $a$ varies.}
  \label{figure_TV_a}
\end{subfigure}\hfill
\begin{subfigure}{.48\textwidth}
  \centering
  \includegraphics[width=\linewidth]{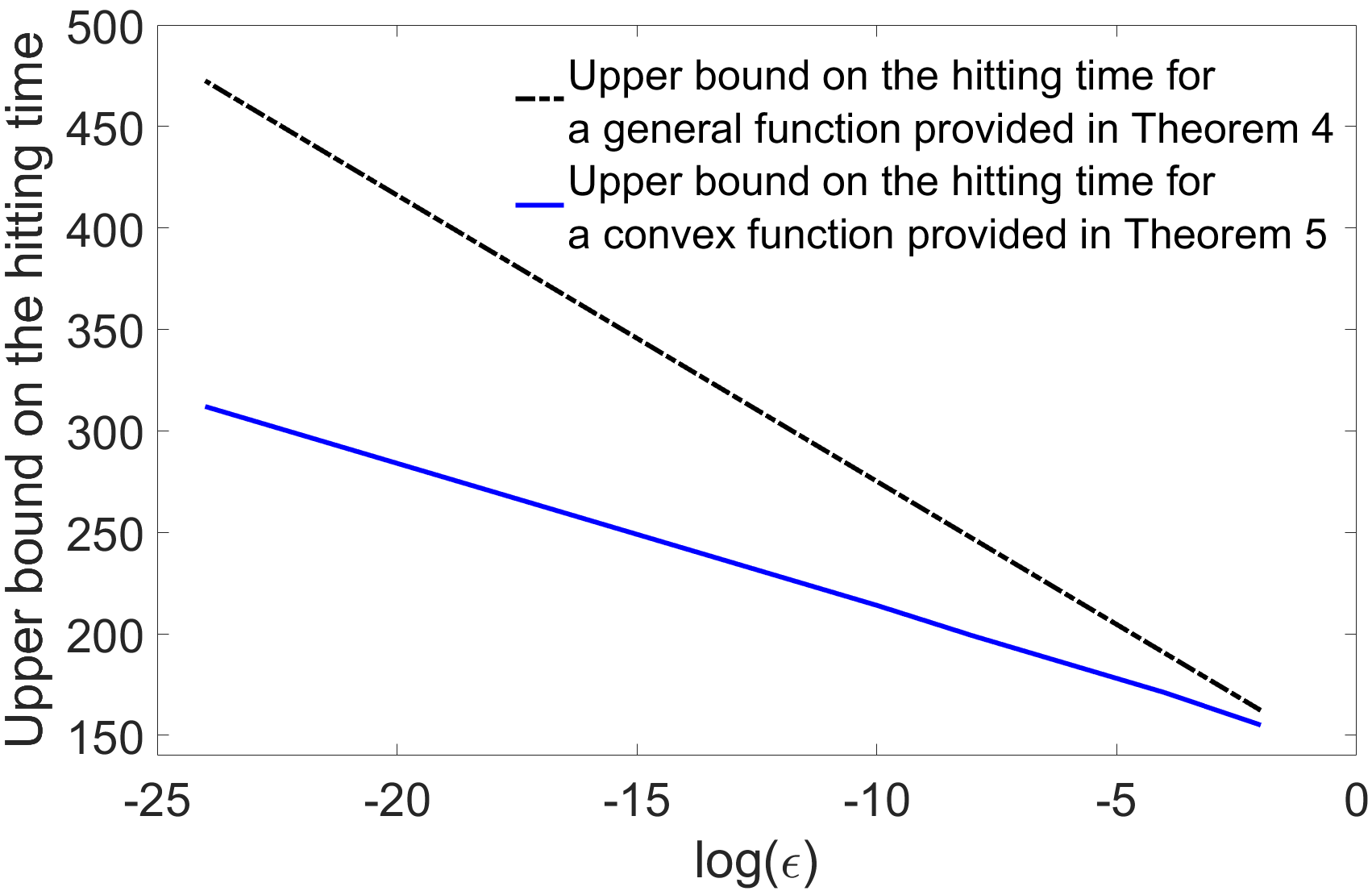}
  \caption{$a = 0.01$ and $\epsilon = L$ varies.}
  \label{figure_TV_b}
\end{subfigure}
\caption{A comparison of the upper bounds in Theorems \ref{theorem_hitting_time} and \ref{theorem_hitting_time_convex} when $M = K = 16, k = 2 \times 10^{-2}$, and $d = 2$. In Figure \ref{figure_TV_b}, the value of $n$ {depends} on $\epsilon$, which is taken into account for drawing the plots.}
\label{figure_TV_epsilon}
\end{figure}

\end{remark}

%The comparison of the Upper bound for hitting time of a general function and a convex function with a set of specified parameters is depicted in Figure \ref{figure_TV_a}.

%\begin{figure}[t]
%\centering
%\includegraphics[width=0.5\textwidth]{TV_a}
%\caption{$n = 3.2 \times 10^{11}, \epsilon = L = 10^{-4}, M = K = 16, k = 2 \times 10^{-2}$, and $d = 2$.}
%\label{figure_TV_a}
%\end{figure}

%\begin{figure}[t]
%\centering
%\includegraphics[width=0.5\textwidth]{TV_n}
%\caption{$a = 0.01, M = K = 16, k = 2 \times 10^{-2}$, and $d = 2$.}
%\label{figure_TV_a}
%\end{figure}

%\begin{figure}[t]
%\centering
%\includegraphics[width=0.5\textwidth]{TV_epsilon}
%\caption{$a = 0.01, M = K = 16, k = 2 \times 10^{-2}$, and $d = 2$.}
%\label{figure_TV_epsilon}
%\end{figure}

%% file: sec/discrete.tex
\section{The Hitting Time Analysis for Discrete Functions}
\label{sec:discrete}

In this section, two variants of  {stochastic time-varying models are studied for discrete functions}.
In the first model, an unknown discrete function is observed with additive noise whose {estimation} function changes over time due to the {presence} of noise.
In the second model, a time-varying linear model with additive noise is studied.

\subsection{Optimization of Functions with Additive Noise}
\label{sec:problem_statement2}

Consider an unknown discrete function $f: \mathcal{X} \rightarrow \mathcal{R}$, where $\mathcal{X} \subset \mathbb{Z}^d$ is a bounded subset of $d$ integer tuples and $\mathcal{R} \subset \mathbb{R}$ is a subset of real numbers ($\mathbb Z$
%and $\mathbb R$
denotes the set of integer
%and real
numbers).
Denote the strict local minima and maxima, known collectively as strict local extrema, of the unknown function $f$ by $\mathcal{X}^*$ defined as
\begin{equation}
    \label{local_extrema_discrete}
    \begin{aligned}
    \mathcal{X}^* = & \{x^*  \in \mathcal{X} : f(x^*) < f(x), \forall x \in \mathcal{B}(x^*) \} \cup \{x^*  \in \mathcal{X} : f(x^*) > f(x), \forall x \in \mathcal{B}(x^*) \}
\end{aligned}
\end{equation}
where $\mathcal{B}(x^*) = \cup_{j = 1}^d \{{x^* + h_j, x^* - h_j}\} \cap \mathcal{X}$ with ${h_1, \dots, h_d}$ being the standard basis of $\mathbb{Z}^d$.
The goal is to find $\mathcal{X}^*$, the set of strict local extrema of the unknown function $f$.
Although the function $f$ is unknown, inquiries of the function values at points in the domain can be made in consecutive rounds, which are evaluated with added noise signals that are mean zero, independent and identically distributed over time and over $\mathcal{X}$.
Formally speaking, the revealed values of the target function $f$ at round $t \in \{1, 2, \dots\}$ are
\begin{equation}
\label{eq_noise2}
    f_t(x) = f(x) + N_t(x), \quad \forall x \in \mathcal{X},
\end{equation}
where $N_t(x)$ are  {noise signals satisfying Assumption \ref{assump:noise_N}.}
% i.i.d. random variables that are strictly bounded by an interval with length $L_N$ and $\mathbb{E}[N_t(x)] = 0$  for all $t \in \{1, 2, \dots\}$ and $x \in \mathcal{X}$.
% In order ,  considered to be different so that their noisy values 
Note that if the noise is disruptive enough, a single set of observed noisy function values $f_t(x)$ for all $x \in \mathcal{X}$ may not represent the unknown target function accurately, making it impossible to find local extrema of the function.
{To address this issue, we estimate} the target function $f$ at round $t - 1$ {by leveraging} the new observations at round $t \in \{2, 3, \dots\}$ as
\begin{equation}
\label{eq_update2}
    \widehat{f}_t(x) = \frac{t - 1}{t} \cdot \widehat{f}_{t - 1}(x) + \frac{1}{t} \cdot f_t(x), \quad \forall x \in \mathcal{X}.
\end{equation}
Note that the {estimation} function $\widehat{f}_t(x)$ changes over time and may not represent the shape of the unknown target function $f$ when $t$ is small.
However, there {may exist a} hitting time $T$ after which  {the {estimation} function $\widehat{f}_t$ shares the same set of local extrema as the target function $f$ with an associated confidence level $1 - a$, where $0 < a \leq 1$.}
% optimizing the  determines the local extrema of the target function 
As a result, the complexity of finding {the} local extrema of the target function $f$ may be irrelevant to the complexity of finding the local extrema of function $\widehat{f}_t$ before {the} hitting time $T$.
Consequently, the complexity of finding the local extrema of the unknown target function $f$ is related to the hitting time $T$ as well as the computational complexity of optimizing function $\widehat{f}_T$.
{Denote} the set of strict local extrema of $\widehat{f}_t$ by $\widehat{\mathcal{X}}^*_t$, {defined as}
\begin{equation}
    \label{estimated_local_extrema_discrete}
    \begin{aligned}
    \widehat{\mathcal{X}}^*_t = & \left \{\widehat{x}^*  \in \mathcal{X} : \widehat{f}_t(\widehat{x}^*) < \widehat{f}_t(x), \forall x \in \mathcal{B}(\widehat{x}^*) \right \} \cup  \left \{\widehat{x}^*  \in \mathcal{X} : \widehat{f}_t(\widehat{x}^*) > \widehat{f}_t(x), \forall x \in \mathcal{B}(\widehat{x}^*) \right \}.
    \end{aligned}
\end{equation}
%the hitting time $T(a)$ is defined below.

\begin{definition}
\label{hitting_time_discrete_added_noise}
{Given $a \in (0, 1],$} the hitting time $T(a)$ for an unknown discrete function $f$ is defined as
\begin{equation}
\label{hitting_time2}
\begin{aligned}
    T(a) = \min \left \{T: \mathbb{P} \left ( \widehat{\mathcal{X}}^*_t = \mathcal{X}^* \right ) \geq 1 - a, {\ \forall t \geq T } \right \},
\end{aligned}
\end{equation}
where $\mathcal{X}^*$ and $\widehat{\mathcal{X}}^*_t$ are defined in \eqref{local_extrema_discrete} and \eqref{estimated_local_extrema_discrete}, respectively.
%, and $\mathbb{P}(\cdot)$ takes the probability of the input event.
\end{definition}
%  {Without loss of generality, we assume }

The hitting time $T(a)$ depends on the minimum distance of the function values of $f$ at point $x \in \mathcal{X}$ from the function values at its neighbor points. {This distance, denoted by} $\delta(x)$, {is defined as}
\begin{equation}
\label{delta_x_value}
    \delta(x) = \min_{x' \in \mathcal{B}(x)}  {| f(x) - f(x') |}.
\end{equation}
% where $\| \cdot \|$ is the $L^2$-norm throughout this {section}.

 {In order to simply the analysis, we make the following assumption about the target function $f$.
\begin{assumption}\label{assump:minimum_distance}
The minimum distance $\delta(x)$ of function $f$ is uniformly lower-bounded by a positive number for all $x\in \mathcal{X}$, i.e., $\delta_m = \min_{x \in \mathcal{X}} \delta(x) > 0$.
\end{assumption}
Intuitively, Assumption \ref{assump:minimum_distance} ensures that function values of $f$ at adjacent points are different, so that their noisy values become distinguishable after enough observations.}
% As stated earlier, it is assumed that $\delta(x) > 0$ for all $x \in \mathcal{X}$, which implies .
The following theorem presents an upper bound {on the} hitting time $T(a)$.

\begin{theorem}
\label{theorem_discrete_upper_bound}
Consider the time-varying function  {$\widehat{f}_t$} in \eqref{eq_update2}.
 {Under Assumptions \ref{assump:noise_N} and \ref{assump:minimum_distance}, 
given $a\in (0,1]$, the associated hitting time $T(a)$} defined in \eqref{hitting_time2}, {satisfies the inequality}
\begin{equation}
\label{UB_1}
    T(a) \leq \frac{2L_N^2}{\delta_m^2} \cdot \ln \left ( \frac{2|\mathcal{X}|}{a} \right ),
\end{equation}
 {where $|\mathcal{X}|$ denotes the number of elements in the set $\mathcal{X}$.}
%that depends on the dimension $d$, and $a$ is the complement of the confidence level used in the definition of hitting time in Definition \ref{hitting_time_discrete_added_noise}.
\end{theorem}
\vspace*{2mm}

\begin{proof}
In order to find an upper bound on the hitting time $T(a)$, note that the hitting event used in \eqref{hitting_time2} satisfies the condition
\begin{equation}
\label{eq_subset2}
    \begin{aligned}
        \left \{ \frac{1}{T} \cdot \Big \| \sum_{t = 1}^{T} N_t(x) \Big \| < \frac{\delta(x)}{2}, \ \forall x \in \mathcal{X} \right \}
        \subseteq \left \{ \widehat{\mathcal{X}}^*_T = \mathcal{X}^* \right \}.
    \end{aligned}
\end{equation}
The above equation holds because \eqref{eq_noise2} and \eqref{eq_update2} result in $\widehat{f}_T(x) = f(x) + \frac{1}{T} \cdot \sum_{t = 1}^{T} N_t(x)$, and if the magnitude of the {noise added} to the true value of function $f$ at point $x$ is less than  {$\delta(x)/2$} for all $x \in \mathcal{X}$, {then} the set of local extrema of the function $\widehat{f}_T$ coincides with {the} set $\mathcal{X}^*$, the local extrema of function $f$.
The probability of the event on the left-hand side of \eqref{eq_subset2} can be  {lower-bounded} as
\begin{equation}
\label{eq_left_hand_side2}
    \begin{aligned}
        \mathbb{P} \left \{ \frac{1}{T} \cdot \Big \| \sum_{t = 1}^{T} N_t(x) \Big \| < \frac{\delta(x)}{2}, \ \forall x \in \mathcal{X} \right \} 
        \overset{(a)}{=} & \ \prod_{i = 1}^{|\mathcal{X}|} \mathbb{P} \left \{ \frac{1}{T} \cdot \Big \| \sum_{t = 1}^{T} N_t(x) \Big \| < \frac{\delta(x)}{2} \right \} \\
        \overset{(b)}{\geq} & \ \prod_{i = 1}^{|\mathcal{X}|}   \left (   1   -   2 \exp   \left (   - \frac{T \delta(x)^2}{2L_N^2} \right )   \right )  \\
        >&   1 - 2 \sum_{i = 1}^{|\mathcal{X}|} \exp \left ( - \frac{T \delta(x)^2}{2L_N^2} \right ) \\
        \geq & \ 1 - 2 |\mathcal{X}| \cdot \exp \left ( - \frac{T \delta_m^2}{2L_N^2} \right ),
    \end{aligned}
\end{equation}
where $(a)$ holds because the added noise signals are independent from each other and $(b)$ follows from Hoeffding's inequality.
Putting \eqref{eq_subset2} and \eqref{eq_left_hand_side2} together, we have
\begin{equation}
\label{eq_lower_bound_prob2}
    \begin{aligned}
        \mathbb{P} \left \{ \widehat{\mathcal{X}}^*_T = \mathcal{X}^* \right \} > 1 - 2 |\mathcal{X}| \cdot \exp \left ( - \frac{T \delta_m^2}{2L_N^2} \right ).
    \end{aligned}
\end{equation}
If $1 - 2 |\mathcal{X}| \cdot \exp \left ( - \frac{T \delta_m^2}{2L_N^2} \right ) \geq 1 - a$ or equivalently $T \geq \frac{2L_N^2}{\delta_m^2} \cdot \ln \left ( \frac{2|\mathcal{X}|}{a} \right ) $, we have  {$\mathbb{P} \left \{ \widehat{\mathcal{X}}^*_T = \mathcal{X}^* \right \} > 1 - a$,} from which the upper bound in \eqref{hitting_time2} follows.
%\begin{flushright}
%\Halmos
%\end{flushright}
\end{proof}

\subsection{A Special Case for Unimodal Functions}
\label{sec:hitting_time_analysis_convex2}

A function $f$ over a bounded set $\mathcal{X} \subset \mathbb{Z}$ is {called} unimodal if it {has only} one global minimum $x^*\in \mathcal{X}$ {and} $f(i) > f(j)$ for all $i<j \leq x^*$, {$i, j \in \mathcal{X}$}, while $f(i) < f(j)$  for all $x^*\leq i<j$.
Assume that the unknown target function $f$ is unimodal over $\mathcal{X}$, which implies it has a single global minimum.
As mentioned earlier, the time-varying function $\widehat{f}_t$ may not even be unimodal
for small values of $t$ under disruptive noise, and therefore it could have multiple local extrema.
However, the single global minimum of the function $f$ becomes known after the hitting time with an associated confidence level.
In this section, a new notion of hitting time is proposed for unimodal functions that captures the complexity of finding the global minimum of the function and does not take the local extrema of the estimated function $\widehat{f}_t$ into account.

 {Without loss of generality, we additionally assume that the noise signals $N_t(x)$ are continuous random variables.
This implies that the estimation function $\widehat{f}_t$ has a single global minimum with probability $1$.
}
{Let} $\widehat{x}_t^* = \argmin_{x \in \mathcal{X}} \ \widehat{f}_t(x)$ {denote the global minimum}.
The hitting time for a unimodal function $f$ is defined {below}.
\begin{definition}
\label{hitting_time_discrete_unimodal}
{Given $a \in (0, 1],$} the hitting time $T_u(a)$ for {a} unimodal function $f$ with its global minimum at $x^* = \argmin_{x \in \mathcal{X}} \ f(x)$ and {its estimated global minimum} $\widehat{x}_t^* = \argmin_{x \in \mathcal{X}} \ \widehat{f}_t(x)$ is defined as
\begin{equation}
\label{hitting_time_convex2}
\begin{aligned}
    T_u(a) = \min \left \{T: \mathbb{P} \big ( \widehat{x}_t^* = x^* \big ) \geq 1 - a, \ \forall t \geq T \right \}.
\end{aligned}
\end{equation}
\end{definition}

The distance of the function value at point $x \in \mathcal{X}$ from the minimum function value is denoted by $\Delta(x)$, which is defined as
\begin{equation}
    \Delta(x)   =  
    \begin{cases}
      f(x) - f(x^*), \quad & \text{if } x \in \mathcal{X} \setminus \{x^*\}, \\
      \min \{ f(x^*   -   1)   -   f(x^*), f(x^*   +   1)   -   f(x^*) \}, &   \text{if }   x   =   x^* .
  \end{cases}
\end{equation}

The following theorem presents an upper bound {on the} hitting time for a unimodal function.

\begin{theorem}
\label{theorem_unimodal_discrete}
Consider the time-varying function  {$\widehat f_t$} defined in \eqref{eq_update2} with $f$ being a unimodal function.
 {Suppose that Assumptions \ref{assump:noise_N} and \ref{assump:minimum_distance} hold.
Given $a\in (0,1]$, the associated hitting time $T_u(a)$}
%, defined in Equation \eqref{hitting_time_convex2},
satisfies the inequality $T_u(a) \leq T$, where $T$ is the smallest number such that
\begin{equation}
\label{UB_2}
    \exp \left ( - \frac{\delta_m^2 T}{2L_N^2} \right ) + 2 \sum_{\substack{i \in \left [ \left \lfloor  {{|\mathcal{X}|}/{2}} \right \rfloor \right ] }} \exp \left ( - \frac{i^2 \delta_m^2 T}{2L_N^2} \right ) \leq a.
\end{equation}
% (Such a number $T$ exists because the left-hand side approaches 0 when $T\to\infty$.)
\end{theorem}

\begin{proof}
By construction, we have $\Delta(x) > 0$ for all $x \in \mathcal{X}$.
In order to find an upper bound on the hitting time $T_u(a)$, {note that} the hitting event used in \eqref{hitting_time_convex2} satisfies the condition
\begin{equation}
\label{eq_subset_convex12}
    \begin{aligned}
        \bigg \{ & \frac{1}{T} \cdot \sum_{t = 1}^{T} N_t(x) > - \frac{\Delta(x)}{2}, \forall x \in \mathcal{X} \setminus \{ x^* \} \textbf{ and } \frac{1}{T} \cdot \sum_{t = 1}^{T} N_t(x^*) < \frac{\Delta(x^*)}{2} \bigg \} \subseteq \Big \{ \widehat{x}_T^* = x^* \Big \}.
    \end{aligned}
\end{equation}

 {Denote the event on the left-hand side of \eqref{eq_subset_convex12} as $E_t$, whose probability can be lower-bounded as}
% The probability of the event on the left-hand side of Equation  can be  {lower-bounded} as
\begin{equation}
\label{eq_left_hand_side_convex12}
{\small
    \begin{aligned}
         {  \mathbb{P}\{E_t\}} &{\overset{(a)}{=}}  \ \mathbb{P}\left \{ \frac{1}{T} \cdot \sum_{t = 1}^{T} N_t(x^*) < \frac{\Delta(x^*)}{2} \right \} \times  \prod_{ \substack{x \in \mathcal{X} \setminus \{x^*\} }} \mathbb{P} \left \{ \frac{1}{T} \cdot \sum_{t = 1}^{T} N_t(x) > - \frac{\Delta(x)}{2} \right \} \\
        &{\overset{(b)}{\geq}}    \left (   1   -   \exp   \left (   - \frac{T \Delta(x^*)^2}{2L_N^2} \right )   \right )   \times \hspace{-4mm} \prod_{\substack{x \in \mathcal{X} \setminus \{x^*\} }}   \left (   1   -   \exp   \left (   - \frac{T \Delta(x)^2}{2L_N^2} \right )   \right ) \\
        &{>}  \ 1 - \exp \left ( - \frac{T \Delta(x^*)^2}{2L_N^2} \right )  - \sum_{\substack{x \in \mathcal{X} \setminus \{x^*\}}}  \exp \left ( - \frac{T \Delta(x)^2}{2L_N^2} \right ) \\
        &{\overset{(c)}{\geq}}  \ 1  -  \exp \left (  -  \frac{T \delta_m^2}{2L_N^2} \right )  -  \sum_{\substack{x \in \mathcal{X} \setminus \{x^*\}}}  \exp \left ( - \frac{T (x - x^*)^2 \delta_m^2}{2L_N^2} \right ) \\
        &{\overset{(d)}{\geq}}  \ 1  -  \exp \left ( - \frac{T \delta_m^2}{2L_N^2} \right ) - 2  \sum_{\substack{i \in \left [ \left \lfloor  {{|\mathcal{X}|}/{2}} \right \rfloor \right ] }}  \exp \left ( - \frac{T i^2 \delta_m^2}{2L_N^2} \right )
\end{aligned}}
\end{equation}
where $(a)$ holds true by the independence property of the added noise signals, $(b)$ is due to Hoeffding's inequality, $(c)$ is true because function $f$ is unimodal, $\Delta(x^*) \geq \delta_m$, and $\Delta(x) \geq (x - x^*) \delta_m$, and $(d)$ results from minimizing the equation with respect to  {all possible values of $x^*$, which gives rise to $x^* = \left \lceil |\mathcal{X}|/{2} \right \rceil$} (taking the ceiling corresponding to the summation through $\left \lfloor  {{|\mathcal{X}|}/{2}} \right \rfloor$).
Putting \eqref{eq_subset_convex12} and \eqref{eq_left_hand_side_convex12} together concludes the proof.
%\begin{flushright}
%\Halmos
%\end{flushright}
\end{proof}
\begin{remark}
 {A number $T$ that satisfies \eqref{UB_2} must exists because the left-hand side of \eqref{UB_2} approaches $0$ when $T\to\infty$.
Also, by substituting the bound in \eqref{UB_1} into \eqref{UB_2}, it can be verified that Theorem \ref{theorem_unimodal_discrete} provides a better bound than Theorem \ref{theorem_discrete_upper_bound} as the properties of unimodal functions are leveraged.}
A comparison of the results of Theorems \ref{theorem_discrete_upper_bound} and \ref{theorem_unimodal_discrete} {along with the details of the simulation model is} depicted in Figure \ref{fig:unimodal_vs_general1}.
%, {where the details of the simulation models are given.}
\end{remark}

\begin{figure}
\centering
\begin{subfigure}{.48\textwidth}
  \centering
  \includegraphics[width=\linewidth]{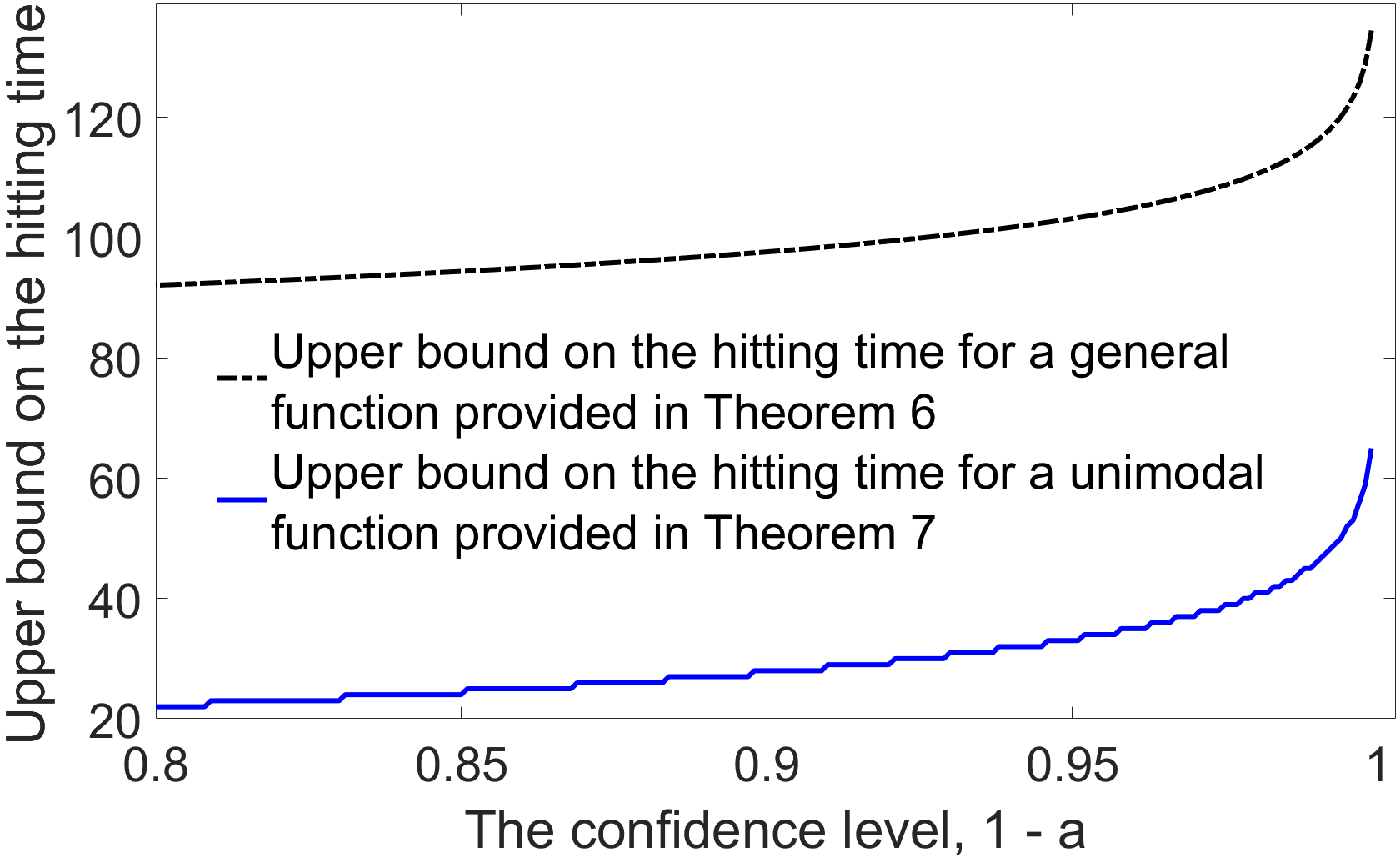}
  \caption{$|\mathcal{D}| = 10000$ and $a$ varies.}
  \label{fig:sub1}
\end{subfigure}\hfill
\begin{subfigure}{.48\textwidth}
  \centering
  \includegraphics[width=\linewidth]{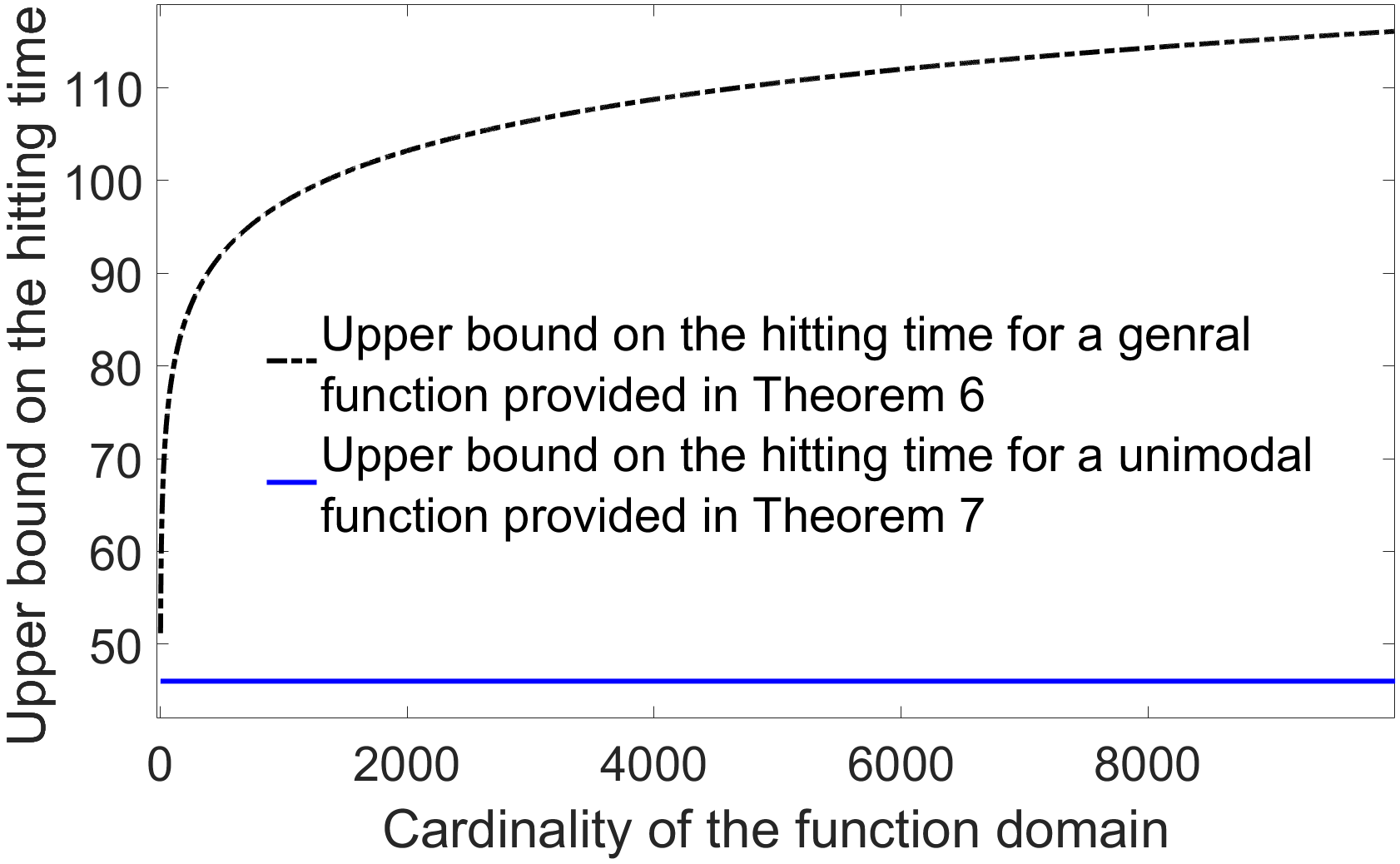}
  \caption{$a = 0.01$ and $|\mathcal{D}|$ varies.}
  \label{fig:sub2}
\end{subfigure}
\caption{A comparison of the upper bounds in Theorems \ref{theorem_discrete_upper_bound} and \ref{theorem_unimodal_discrete} when $L_N = 0.02$ and $\delta_m = 0.01$.}
\label{fig:unimodal_vs_general1}
\end{figure}

\subsection{Time-Varying Linear Model with Additive Noise}
\label{sec:linear-model}
%The additive noise model studied in  Sections~\ref{sec:problem_statement2} provides valuable information about the hitting time, but
% Although the time-variation of functions that arise in many real-world problems are of a nonlinear nature, we argue the generality of a linear model of time-variation, which is the basis of our study of hitting time  to follow.
 {In this section, we study a linear model of time-variation and analyze the hitting time under shape-dominant operators.}
% Recall the standard fact in linear algebra that for any vectors $x, y \in \bR^d$, there exists an affine transformation that satisfies $y=Ax + b$, and if $x\neq0$, there exists a linear transformation that satisfies $y = Ax$. 
 {Consider} the Hilbert space $L^2(\mathcal{X})$, where the inner product of $f$ and $g \in L^{2}(\mathcal{X})$ is defined by $\langle f, g\rangle = \int_{\mathcal{X}} f(x) g(x) dx$. 
We use the same inner product notation when  {the domain $\mathcal{X}$ is a discrete set}. For any nonzero functions $f, g\in L^2$, there exists a bounded linear transformation $\mathcal{T}: L^2(\mathcal{X}) \to L^2(\mathcal{X})$ such that $\mathcal{T} f = g$. In fact, one such transformation is given by $\mathcal{T} h = \frac{\langle f, h\rangle }{\langle f, f\rangle} g$. Since the zero function is trivial to optimize, the restriction to linear transformation is a general framework that captures the varying nature of nonlinear functions.
%in the examples of Section~\ref{sec:introduction}.

We further note that for any scalar $\lambda > 0$, the functions $f$ and $\lambda f$ share the same set of local minima. Rescaling by a positive number does not affect the complexity of the optimization problem. Hence, restricting the linear operators $\mathcal{T}$ to have norm $1$ incurs no loss of generality.

In practice, the functions to be minimized are often not specified exactly, due to the rounding error of numerical computation or the inexact nature of the model. We model this limitation by  {the} random perturbation $w$ sampled from some distribution.
Given a sequence of linear operators  {$\{\mathcal{A}_t\}$}
% $\mathcal{A}_0, \mathcal{A}_1, \ldots, \mathcal{A}_{t-1}$ 
such that $\norm{\mathcal{A}_{ {t}}} = \sup_{f \neq 0} \frac{\norm{\mathcal{A}_{ {t}} f}}{\norm{f}} = 1$ together with the perturbations  {$\{w_t\}$}, consider the following model of linear time variation:
\begin{equation}\label{eq:noisy_linear_operator}
f_{t+1} = \mathcal{T}_t f_t = \mathcal{A}_t f_t + w_t, \quad \text{for } t \in \{0, 1, \dots\}.
\end{equation}
\emph{What properties the operators  {$\{\mathcal{T}_t\}$} should satisfy in order for $f_{t}$ to almost reach a target function $f^*$ at time  {$t=T$}?}
We will provide an answer using the notion of shape dominant operator.
To understand the importance of this problem,
suppose that at time $t=0$, we optimize $f_0$ around a poor local minimum $x_0^*$. If at $t= {T}$, the function $f_{ {T}}$ becomes convex with a unique global minimum $x_{ {T}}^*$, then no matter how optimization is carried out for $f_1$ through $f_{ {T}-1}$, minimizing $f_{ {T}}$ will yield the same solution $x_{ {T}}^*$, which is globally optimal. The effect of minimizing $f_{ {T}}$ cancels out the sub-optimality at time  {$t=0$}.
Moreover, under some technical conditions, the global solution at time ${ {T}}$ can be used to find global solutions at future times using tracking methods~\cite{ding2021escaping, fattahi2020absence, massicot2019line}. In other words, the shape of $f_{ {T}}$ affects the complexity of online optimization in the long run.
% \subsubsection{Shape Dominant Model}
% \label{sec:shape_dominante_model}

 {Now, we introduce the notion of shape dominant operator.
Consider time-varying functions $\{f_t\}$ defined on a finite discrete set $\mathcal{X}= \{x_1, \ldots, x_n\}\subset \mathbb{Z}^d$. Equivalently, $f_t$ can be viewed as a vector in $\bR^{n}$.
For the noisy linear operator $\mathcal{T}_t$ defined in \eqref{eq:noisy_linear_operator}, let $A_t$ denote the associated matrix of the linear operator $\mathcal{A}_t$ represented under the standard basis, for $t\in\{1,2,\dots\}$.
Let $P(A_t, w_t)$ denote the joint distribution of $A_t$ and $w_t$.
}
\begin{definition}\label{def:shape-dominant}
The joint distribution $P(A, w)$ is said to be $(\delta, \sigma, f^*, \phi^*)$ shape dominant if following conditions hold with probability $1$:  {1) the unit vector $f^*$ is the eigenvector of $A$ associated with eigenvalue $1$; 2) the unit vector $\phi^*$ is the eigenvector of $A^\top$ associated with eigenvalue $1$; 3) $\langle f^*, \phi^*\rangle \neq 0$; 4) all other eigenvalues of $A$ have  {absolute values} less than $1-\delta$; 5) conditioned on $A$, the noise $w$ has zero mean and is sub-Gaussian with parameter $\sigma^2$ in the sense that for all $u\in \bR^{n}$ with $\norm{u} \leq 1$,  {it holds that} $\bE[\exp(s u^\top w)] \leq \exp\left(\frac {\sigma^2s^2}{2}\right)$.}
\end{definition}

% Temperately delete for more space.
% In order to understand the conditions in Definition~\ref{def:shape-dominant}, consider the special case where $A$ is a positive stochastic matrix whose column sums are all $1$. The unit vector $\phi^* = (\frac{1}{\sqrt{n}}, \ldots, \frac{1}{\sqrt{n}})$ is the eigenvector of $A^\top$ associated with eigenvalue $1$.
% %  {, which validates Condition 2.} 
% By the Perron-Frobenius theorem, $A$ also has an all-positive eigenvector $f^*$ with eigenvalue $1$, and all other eigenvalues of $A$ have  {absolute values} strictly less than $1$. 
%  {Therefore, Conditions 1 - 4 are satisfied. The vector $f^*$ is the equilibrium distribution of a Markov chain whose transition matrix is $A$.}
\begin{theorem}\label{thm:shape-dominant}
 {For the time-varying operator $\mathcal{T}_t$ defined in  \eqref{eq:noisy_linear_operator}, suppose that $P(A_t, w_t)$ is $(\delta, \sigma_t, f^*, \phi^*)$ shape dominant and independent for all $t \in \{0,1, \dots, T-1 \}$,} then,
\begin{align}
 {f_T}  = \frac{\langle \phi^*, f_0 + \sum_{t=0}^{ {T-1}} w_t \rangle}{\langle \phi^*, f^*\rangle} f^* + v + w,
\end{align}
where $\norm{v}\leq (1-\delta)^{ {T}} \left(\norm{f_0} + \frac{\langle \phi^*, f_0\rangle}{\langle \phi^*, f^*\rangle}\right)$ and $w$ is sub-Gaussian with parameter $\sigma^2 = \left( 1 + \frac{1}{\langle \phi^*, f^*\rangle^2} \right) \sum_{t=0}^{ {T-1}}(1-\delta)^{2( {T}-t)}\sigma_t^2$.
\end{theorem}

\begin{proof}
% Consider the operator $\mathcal{T}_i f =A_i f + w_i$ that is $(\delta, \sigma_i, f^*, \phi^*)$ shape dominant for $i \in \{0, 1 \ldots, k-1\}$. 
 {Consider the subspace $\mathcal{G} = \{g \in \bR^{n}, \langle \phi^*, g\rangle=0\}$.}
Since $\langle \phi^*, f^*\rangle \neq 0$, we have $f^* \notin \mathcal{G}$.
Since $\phi^*$ is the eigenvector of  {$A_t^\top$}, the following holds for all $g \in \mathcal{G}$
\begin{align}
\langle \phi^*,  {A_t}g\rangle  = \langle  {A_t}^\top \phi^*, g\rangle = \langle \phi^*, g\rangle = 0.
\end{align}
Therefore, $ {A_t} g \in \mathcal{G}$, and $\mathcal{G}$ is an invariant subspace of  {$A_t$} in $\bR^{n}$  {for $t\in \{0,1, \dots, T-1 \}$}. Let a basis of $\mathcal{G}$ be given by $\{g_1, \ldots, g_{n-1}\}$. Then, $B = \{f^*, g_1, \ldots, g_{n-1}\}$ is a basis of $\bR^{n}$, under which the linear operator  {$A_t$} takes the form
\begin{align}
 {A_t} =
\begin{bmatrix*}
1 & 0 & \ldots & 0 \\
0 \\
\vdots & &  {A_t^\prime} \\
0 \\
\end{bmatrix*}, \label{eq:A-decompose}
\end{align}
where  {$A_t^\prime$} is a random matrix in $\bR^{(n-1) \times (n-1)}$. With a slight abuse of notation, we regard  {$A_t'$} as a linear transformation from $\mathcal{G}$ to $\mathcal{G}$. Note that $ {\norm{A_t'}} \leq 1-\delta$ because all other eigenvalues of  {$A_t$} have norm less than $1-\delta$.
Under the basis $B$, $f_0$ has the representation  {$f_0 =\frac{\langle \phi^*, f_0\rangle}{\langle \phi^*, f^*\rangle} f^* + g$,
where $g\in \mathcal{G}$.} As a result,
\begin{equation}
\begin{aligned}
 {f_T} & = \mathcal{T}_{ {T-1}} \circ \cdots \circ \mathcal{T}_0 f_0\\
& =   A_{ {T-1}} \cdots A_0 f_0 + \sum_{ {t=0}}^{ {T-1}} A_{ {T-1}} \cdots  {A_{t+1} w_t} \\
& =  \frac{\langle \phi^*, f_0\rangle}{\langle \phi^*, f^*\rangle} f^* + A_{ {T-1}}' \ldots A_1' g +  {\sum_{t=0}^{T-1} A_{T-1} \cdots A_{t+1} w_t}.
\end{aligned}
\end{equation}
The norm estimate gives rise to
\begin{equation}
\begin{aligned}
\norm{A_{ {T-1}}'   \ldots   A_1' g}    &\leq   (1 - \delta)^{ {T}}   \cdot   \norm{g}   \leq   (1 - \delta)^{ {T}}   \cdot   \left( \norm{f_0}   +   \abs{\frac{\langle \phi^*, f_0\rangle}{\langle \phi^*, f^*\rangle}}\right) ,
\end{aligned}
\end{equation}
where the triangle inequality is used. Similarly, one can write  {$w_t = \frac{\langle \phi^*, w_t\rangle}{\langle \phi^*, f^*\rangle} f^* + h_t$, where $h_t \in \mathcal{G}$.} We have
\begin{equation}
A_{ {T-1}} \cdots  {A_{t+1} w_t} = \frac{\langle \phi^*,  {w_t}\rangle}{\langle \phi^*, f^*\rangle} f^* + A'_{ {T-1}} \cdots  {A'_{t+1} h_t}.
\end{equation}
For all $u\in \bR^{n}$ with $\norm{u} \leq 1$, it holds that
\begin{equation}
\begin{aligned}
& \quad \, \bE \left [\exp \left (s \left  \langle u, A_{ {T-1}}' \cdots  {A_{t+1}'  h_t} \right \rangle \right ) \right ] \\
& = \bE \left [\exp \left (s \left \langle  {A_{t+1}'^\top} \cdots A_{ {T-1}}'^\top  u,  {h_t} \right \rangle \right ) \right ] \\
& =  \bE \left [\exp \left (s \left \langle A_{ {t+1}}'^\top \cdots A_{ {T-1}}'^\top  u,  {w_t} - \frac{\langle \phi^*,  {w_t}\rangle}{\langle \phi^*, f^*\rangle} f^* \right \rangle \right ) \right ] \\
& = \bE \Bigg[\exp \left (s \left \langle A_{ {t+1}}'^\top \cdots A_{ {T-1}}'^\top  u ,  {w_t} \right \rangle \right ) \times  \exp \left (s \left \langle - \frac{\langle A_{ {t+1}}'^\top \cdots A_{ {T-1}}'^\top  u, f^* \rangle}{\langle \phi^*, f^*\rangle} \phi^*,  {w_t} \right \rangle \right )\Bigg] \\
& \leq \exp\left(\frac {\sigma_{ {t}}^2s^2 \norm{A_{ {t+1}}'^\top \cdots A_{ {T-1}}'^\top u}^2}{2}\right) \times  \exp \left( \frac{\sigma_{ {t}}^2 s^2}{2} \left(\frac{\langle A_{ {t+1}}'^\top \cdots A_{ {T-1}}'^\top  u, f^* \rangle}{\langle \phi^*, f^*\rangle}\right)^2\right)\\
& \leq \exp\left(\frac {\sigma_{ {t}}^2s^2(1-\delta)^{2( {T-t})}\left( 1 + \frac{1}{\langle \phi^*, f^*\rangle^2} \right)}{2}\right),
\end{aligned}
\end{equation}
 {which} implies that $  {A'_{T-1} \cdots A'_{t+1} h_t}$ is sub-Gaussian with parameter $\sigma_{ {t}}^2(1-\delta)^{2( {T-t})}\left( 1 + \frac{1}{\langle \phi^*, f^*\rangle^2} \right)$, and thereby, $ {\sum_{t=0}^{T-1} A_{T-1}' \cdots A_{t+1}' h_t}$ is sub-Gaussian with parameter $\sigma^2 = \left( 1 + \frac{1}{\langle \phi^*, f^*\rangle^2} \right)  {\sum_{t=0}^{T-1}(1-\delta)^{2(T-t)}\sigma_t^2}$. %\qquad \qquad \qquad \qquad \qquad \qquad \qquad \qquad \qquad \qquad \qquad \qquad \qquad \qquad \qquad \qquad \qquad \quad \Halmos
%\begin{flushright}
%\Halmos
%\end{flushright}
 {This completes the proof.}
\end{proof}

Theorem~\ref{thm:shape-dominant} states that if the time-varying model is given by shape dominant operators, the function  {$f_T$} decomposes into the sum of dominating shape $f^*$, a bias term $v$ that gradually fades away, and a cumulating noise term that discounts noise in previous iterations.
We provide a bound {on the} hitting time {below}.
\begin{theorem}\label{thm:hitting}
Under the same assumptions made in Theorem~\ref{thm:shape-dominant},  {for a given $\epsilon >0$, define the associated hitting time $T(\epsilon)$ as}
 \begin{align} \label{eq:hitting-time-def}
         {T(\epsilon) = \min \big \{T: \exists \lambda\in \bR \text{ s.t. }\norm{f_T - \lambda f^*} < \epsilon \big \}}.
    \end{align}
 {Then, for all $T> \frac{\log{2\left(\norm{f_0} +  \abs{\frac{\langle \phi^*, f_0\rangle}{\langle \phi^*, f^*\rangle}}\right)} - \log \epsilon}{\log \frac{1}{1-\delta}}$, it holds that}
{\small{\begin{equation}
 {
\bP(T(\epsilon)   \geq   T)   \leq  
C_n   \exp   \left(   - \frac{\epsilon^2}{32   \left(   1   +   \frac{1}{\langle \phi^*, f^*\rangle^2}   \right)   \sum_{t=0}^{T-1}(1 - \delta)^{2(T-t)}\sigma_t^2} \right) ,
}
\end{equation}}}

\noindent where $C_n$ is a universal constant depending only on $n$.
\end{theorem}

\begin{proof}
 {By Theorem~\ref{thm:shape-dominant}, for a fixed number $T$, we have the following decomposition for $f_T$:}
\begin{equation}
        f_{ {T}} = \frac{\langle \phi^*, f_0 +  {\sum_{t=0}^{T-1} w_t} \rangle}{\langle \phi^*, f^*\rangle} f^* + v^{( {T})} + w^{( {T})},
\end{equation}
where $\norm{v^{( {T})}} < (1-\delta)^{ {T}} \left(\norm{f_0} + \abs{\frac{\langle \phi^*, f_0\rangle}{\langle \phi^*, f^*\rangle}}\right)$ and  {$w^{(T)} = \sum_{t=0}^{T-1} A_{T-1}' \cdots A_{t+1}' h_t$} is sub-Gaussian with parameter $\sigma^2 = \left( 1 + \frac{1}{\langle \phi^*, f^*\rangle^2} \right)  {\sum_{t=0}^{T-1}(1-\delta)^{2(T-t)}\sigma_t^2}$. 
    From the definition of the hitting time $T(\epsilon)$ in \eqref{eq:hitting-time-def}, we have
\begin{equation}
 {\bP(T(\epsilon) <  T) \geq \bP\left(\norm{v^{(T)}} < \epsilon/2, \norm{w^{(T)}} < \epsilon/2\right).}
\end{equation}
    When $ {T}> \frac{\log{2\left(\norm{f_0} +  \abs{\frac{\langle \phi^*, f_0\rangle}{\langle \phi^*, f^*\rangle}}\right)} - \log \epsilon}{\log \frac{1}{1-\delta}}$, the bound $\norm{v^{( {T})}} < \epsilon/2$ is satisfied. Since $w^{( {T})}$ is sub-Gaussian with parameter $\sigma^2$, the tail-bound for $w^{( {T})}$ yields
{\small
\begin{equation}
{ {
\bP   \left(\norm{w^{({T})}}   <   \epsilon/2\right)    =   1   -   \bP\left(\norm{{w^{{(T)}}}}   >   \epsilon/2\right)  \geq   1   -   C_n   \exp \left( -\frac{\epsilon^2}{32 \sigma^2} \right) ,
}}
\end{equation}}

\noindent where $C_n$ is a universal constant depending only on $n$. %\qquad \qquad \qquad \qquad \qquad \qquad \qquad \qquad \Halmos
     {This completes the proof.}
\end{proof}

To understand the above bound, consider a fixed  {time $T$}. When  {$\sigma_t$} decreases, the bound becomes smaller. 
As a result, with a smaller random perturbation, it is more likely to reach the target function faster. 
 {When} $\epsilon$ increases, the bound  {also} becomes smaller, which matches the intuition that a larger neighborhood is easier to reach than a smaller one.

\begin{remark}
 {The analysis in this section}
% of the linear model with additive noise in Section \ref{sec:linear-model}
%and \ref{sec:shape_dominante_model}
can be generalized {to} continuous functions by working through eigenfunctions as opposed to eigenvectors.
%  {Details are provided in the online technical report \cite{yekkehkhanyhitting}.}
We briefly discuss this in the special case where  {${L}^2(\mathcal{X})$} has a finite number of bases.
%if the function space has finite number of bases.
%This generalization is possible by using the link between eigenfunctions and eigenvectors of matrices that is described below \cite{wiki:xxx}.
%Eigenfunctions and linear operators can be expressed as column vectors and matrices, respectively, but they may have infinite dimensions.
Let the inner product be $\langle f , g \rangle = \int_{ {\mathcal{X}}} f(x) \cdot g(x) dx$ and the function space to have an orthonormal basis given by the set of functions $\{ u_1, u_2, \dots, u_n \}$  {such that}
%, where $n$ may be infinite
\begin{equation}
    \label{orthonormal_basis}
    \langle u_i, u_j \rangle = \int_{ {\mathcal{X}}} u_i(x) \cdot u_j(x) dx =
    \begin{cases}
      1 & \text{if } i = j\\
      0 & \text{if } i \neq j
    \end{cases}.
\end{equation}
Note that any function can be decomposed into a linear combination of the basis functions, i.e.,  {$f(x) = \sum_{j = 1}^n a_j \cdot u_j(x)$,} where the coefficients can be stacked into a column vector $a = [a_1, a_2, \dots, a_n]^T$.
Define the matrix $A$ representing the linear operator $\mathcal{T}$ with {the} elements
\begin{equation}
    \label{A_matrix}
    A_{ij} = \langle u_i, \mathcal{T} (u_j) \rangle = \int_{ {\mathcal{X}}} u_i(x) \cdot \mathcal{T} \big (u_j(x) \big ) dx.
\end{equation}
 {There exists a vector $b = [b_1, b_2, \dots, b_n]^T$ such that applying the operator $\mathcal{T}$ on the decomposed form of $f(x)$ yields}
\begin{equation}
    \label{apply_operator}
    \mathcal{T} \big ( f(x) \big )
    = \sum_{j = 1}^n a_j \cdot \mathcal{T} \big ( u_j(x) \big )
    = \sum_{j = 1}^n b_j \cdot u_j(x).
\end{equation}
Taking the inner product of both sides of the above equation with an arbitrary basis function $u_i$ {leads to}
\begin{equation}
    \sum_{j = 1}^n a_j \hspace{-0.25mm} \cdot \hspace{-0.25mm} \big \langle u_i, \mathcal{T} \big ( u_j \big ) \big \rangle
       =  \sum_{j = 1}^n b_j \hspace{-0.25mm} \cdot \hspace{-0.25mm} \langle u_i, u_j \rangle \Rightarrow \sum_{j = 1}^n a_j \hspace{-0.25mm} \cdot \hspace{-0.25mm} A_{ij}  =  b_i.
\end{equation}
The above equation is the matrix multiplication $Aa = b$, which is the matrix  {associated with} $\mathcal{T}$ acting upon the function $f(x)$ expressed in the orthonormal basis.
If $f(x)$ is an eigenfunction of transformation $\mathcal{T}$ with eigenvalue $\lambda$, we have $Aa = \lambda a$.
{Hence, the results of Theorem \ref{thm:hitting} can be applied to continuous functions in a function space with a finite number of bases.}
The extension to the case with an infinite, but countable, number of bases is similar under some technical assumptions.
%In Section \ref{problem_statement}, we study different variants of time-varying probabilistic contraction mappings for the cases where the number of function space bases is not necessarily finite.
%and discuss their applications in non-convex time-varying optimization,
%where not only does the mapping change over time or is probabilistic, but a noise function is also added to the outcome of the mapping in each iteration.
\end{remark}

%% file: sec/simulation.tex
\section{Simulation Results}
\label{sec:simulation}
%Policy iteration is usually slower than value iteration for a large number of possible states.
In this section, the adversarial attack on the computation of value iteration is simulated for an agent interacting with an environment depicted in Figure \ref{fig:ad_VI}.
\begin{figure}
\centering
\begin{subfigure}[t]{.43\textwidth}
  \centering
  \includegraphics[width=\linewidth]{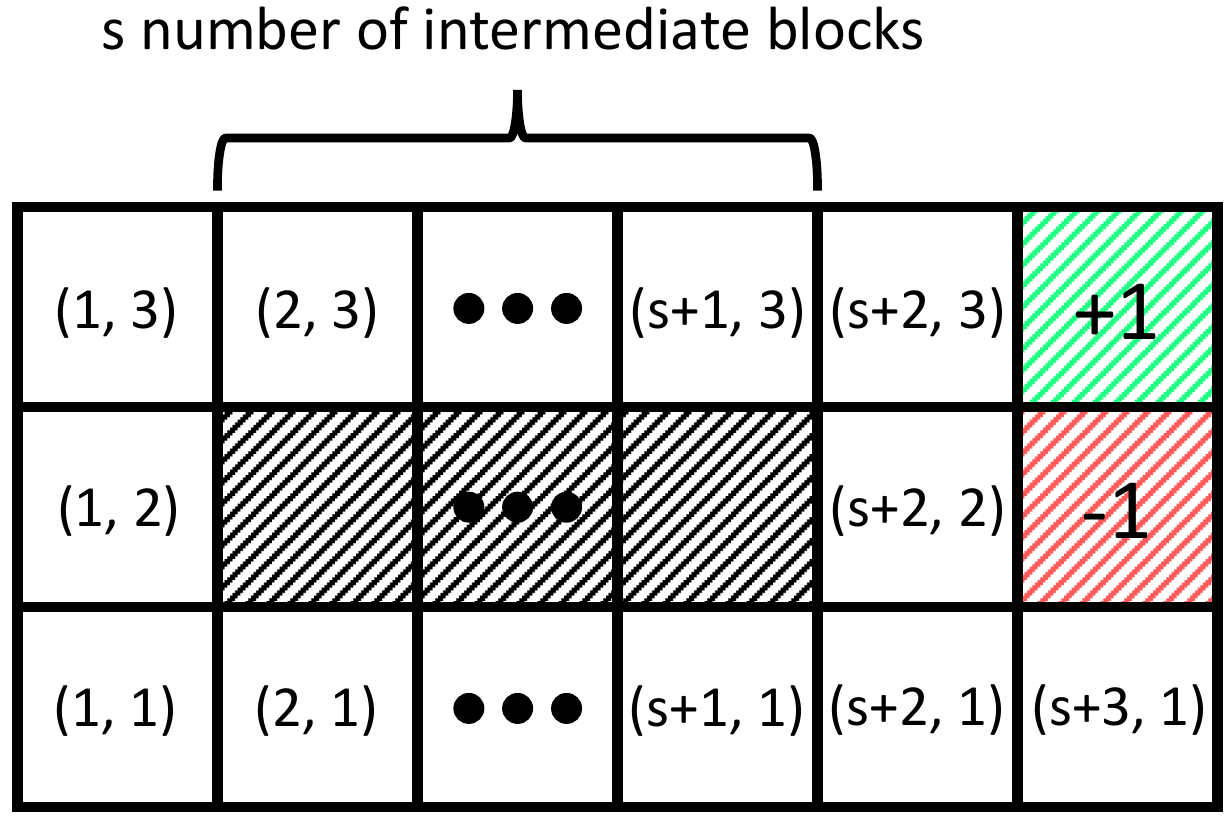}
  \caption{}
  \label{fig:ad_VIb}
\end{subfigure}\hspace{1cm}
\begin{subfigure}[t]{.26\textwidth}
  \centering
  \includegraphics[width=\linewidth]{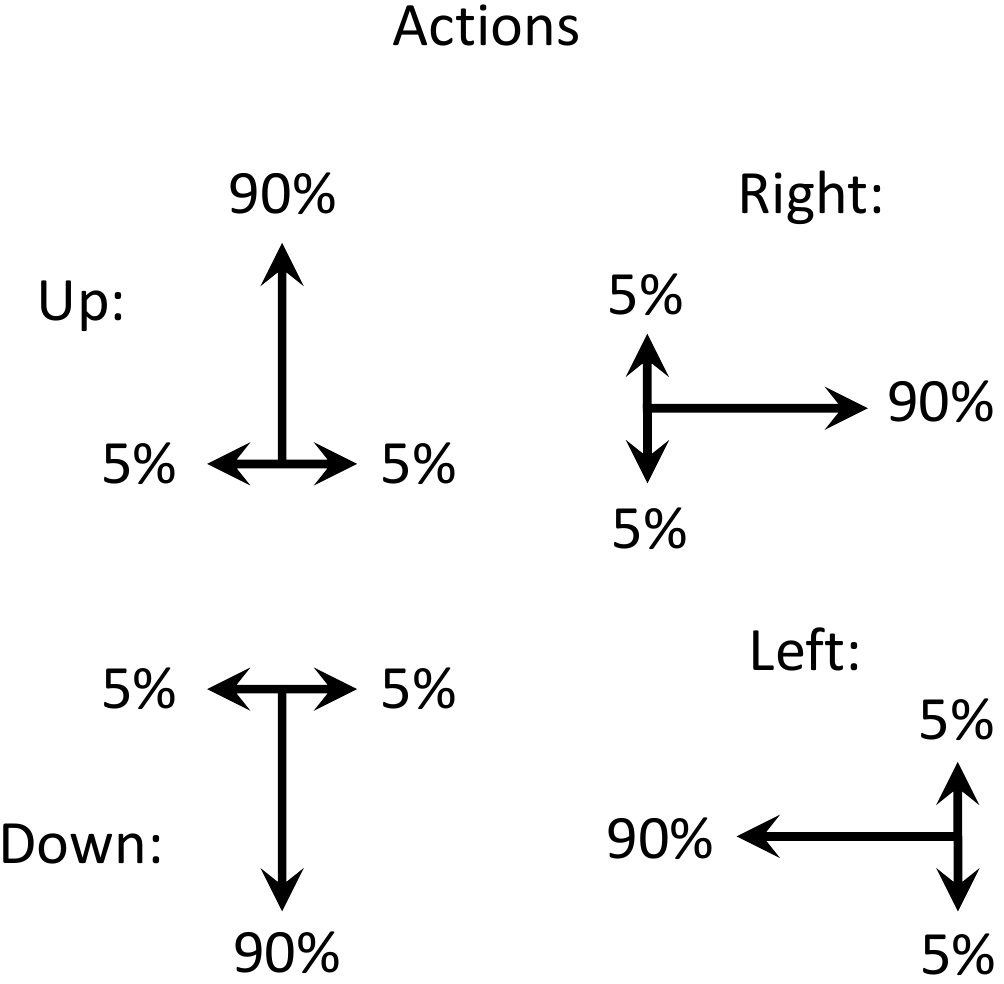}
  \caption{}
  \label{fig:ad_VIa}
\end{subfigure}
\caption{{(a)} the agent interacts with an environment, {(b) the agent} has a set of four actions in each state.}
\label{fig:ad_VI}
\end{figure}
\begin{figure}
\centering
\begin{subfigure}[t]{.488\textwidth}
  \centering
  \includegraphics[width=\linewidth]{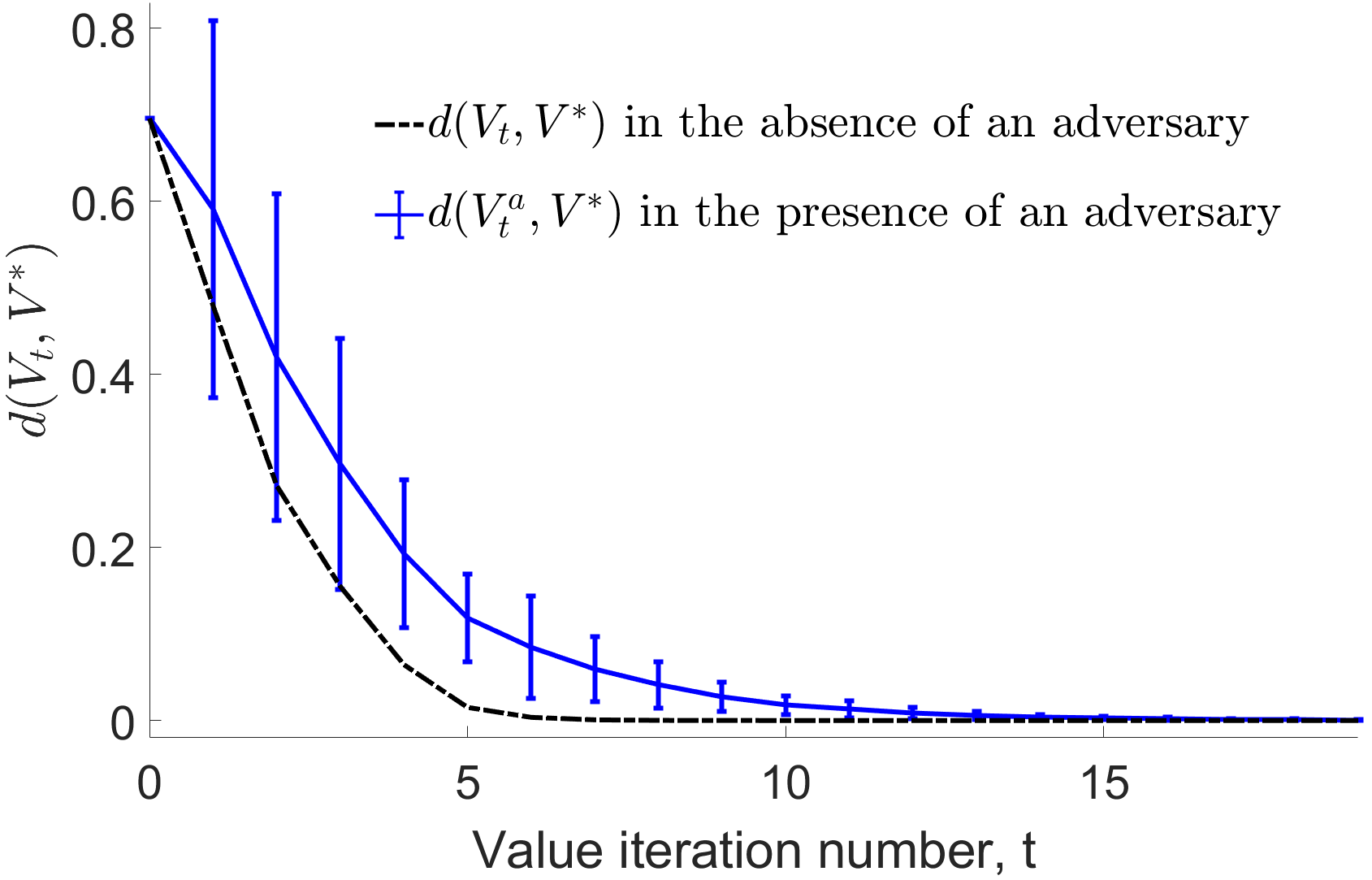}
  \caption{\label{first}A comparison of value iteration convergence in the absence and presence of an adversary.}
  \label{fig:ad_VI_sub1}
\end{subfigure}\hfill
\begin{subfigure}[t]{.488\textwidth}
  \centering
  \includegraphics[width=\linewidth]{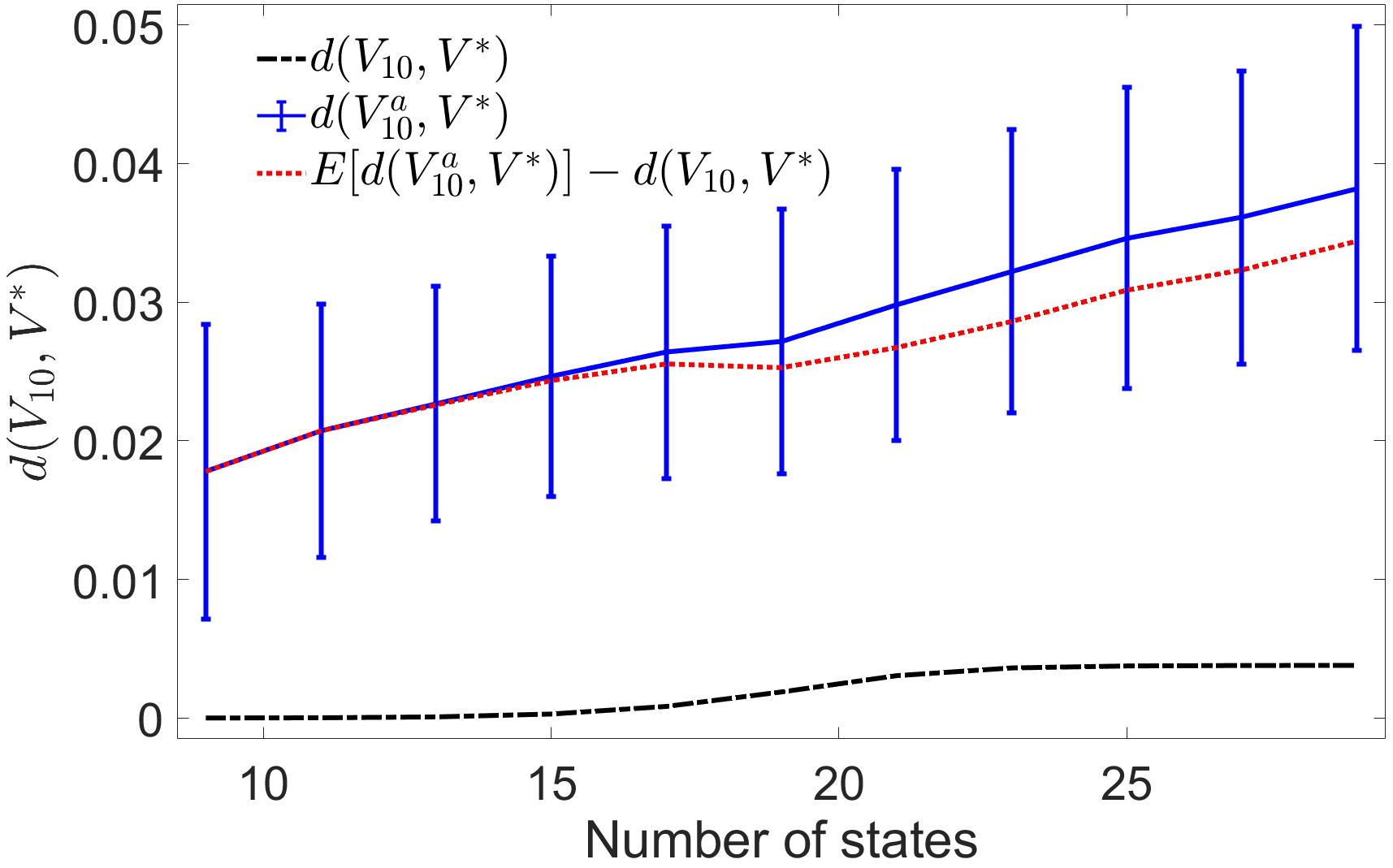}
  \caption{\label{second}The effect of an adversary versus the number of states.}
  \label{fig:ad_VI_sub2}
\end{subfigure}
\caption{The effect of an adversary on the convergence of value iteration.}
\label{fig:unimodal_vs_general}
\end{figure}
The agent can take any of the four actions Up, Down, Right, and Left in each of the non-terminal states.
By taking an action, the agent {moves} one block toward the desired action $90\%$ of the time, or {moves} one block to the right or left of the desired taken action uniformly at random $10\%$ of the time.
The agent bounces back to its original state before taking an action if movement in the direction described above is not possible due to the walls marked with diagonal strips or exiting the environment.
The agent is incurred a cost of $0.02$ by each move and there are two terminal states in which the agent receives an immediate reward of +1 and -1 as shown in Figure \ref{fig:ad_VI}.
In order to determine the optimal path for the agent starting from any of the states, the value function is calculated using synchronous value iteration.
In our simulated example, an adversary contaminates the value function by expanding up to $Q = 1.8$ in a random direction, withholding the contraction, $20\%$ of the time.
As a result, the distance of the time-varying value function from the true value function based on the $L^2$-norm is affected negatively as depicted in Figure \ref{fig:ad_VI_sub1}, where the starting function is the all-zero function in our simulations and the average and standard deviations are estimated by 1000 rounds of independent runs of the value iteration.
Furthermore, the negative effect of the adversary is {worsened} by increasing the cardinality of the state space in the studied example.
In order to show this, the number of intermediate blocks in Figure \ref{fig:ad_VI} is changed from {1} to {10}, i.e., the number of states is changed from {9} to 27, and the distance between the value function at the tenth iterate and the true value function is depicted in Figure \ref{fig:ad_VI_sub2}.
As {shown} in Figure \ref{fig:ad_VI_sub2}, $\mathbb{E} \big [d(V_{10}^a, V^*) \big ] - d(V_{10}, V^*)$ has an increasing trend as the number of states increases, where $V_{10}^a$ is value function at the tenth iterate in the presence of an adversary and $V_{10}$ is the corresponding function in the absence of an adversary, and the dependence of value function on the number of states is eliminated to keep the notations simple.

%% file: sec/conclusion.tex
\section{Conclusion and Future work}
\label{sec:conclusion}

Multiple models of stochastic time variation along with their corresponding notions of hitting time are studied in this {paper}.
In particular, we develop
%new theoretical results,
a probabilistic Banach fixed-point theorem {that proves the convergence of the value iteration method with} a probabilistic contraction-expansion transformation {with an associated confidence level, which} finds applications to adversarial attacks on computation of {the} value iteration {method}.
{We prove that the hitting time of the value function in the value iteration method with a probabilistic contraction-expansion transformation is logarithmic in terms of the inverse of a desired precision.}
Furthermore, we {develop upper bounds on} the hitting time for optimization of unknown discrete and continuous time-varying functions whose noisy evaluations are revealed over time.
{The upper bound for a discrete function is logarithmic in terms of the cardinality of the function domain and the upper bound for a continuous function is super-quadratic (but sub-cubic) in terms of the inverse of a desired precision.}
In this framework, we show that convex functions are learned faster than non-convex functions.
Finally, {an upper bound on} the hitting time is {developed} for a time-varying linear model with additive noise {under} the notion of shape dominance for discrete functions.
{Future research directions include:} studying how an environment with time-varying parameters {modeled by} transition probabilities and rewards affects the Bellman transformation and its fixed point, obtaining upper bounds on the {rate} of change {of the} time-varying parameters such that the time-varying fixed points are achievable after a hitting time{, and studying} the effect of an adversary in applications of reinforcement learning whose computations are performed via edge computing.